\documentclass{amsart}
\usepackage{bbm}
\usepackage{amsmath}
\usepackage{amsfonts}
\usepackage{amssymb}
\usepackage{amsthm}
\usepackage{enumerate}
\usepackage[colorinlistoftodos]{todonotes}
\usepackage{hyperref}
\usepackage{mathrsfs}

\usepackage{array,etoolbox}
\preto\tabular{\setcounter{magicrownumbers}{0}}
\newcounter{magicrownumbers}
\newcommand\rownumber{\stepcounter{magicrownumbers}\arabic{magicrownumbers}}

\theoremstyle{plain}
\newtheorem{theorem}{Theorem}[section]
\newtheorem{lemma}[theorem]{Lemma}
\newtheorem{cor}[theorem]{Corollary}

\newtheorem{proposition}[theorem]{Proposition}
\newtheorem{definition}[theorem]{Definition}
\newtheorem*{theorem*}{Theorem}
\theoremstyle{remark} 
\theoremstyle{remark} 

\newcommand{\bb}{\mathbb}
\newcommand{\mrm}{\mathrm}

\newcommand{\cay}{\mathrm{Cay}}
\newcommand{\Sym}{\mathrm{Sym}}
\newcommand\sym[1]{\Sym({#1})}

\newcommand{\calS}{\mathcal{S}}
\newcommand{\calA}{\mathcal{A}}

\newcommand{\ind}{\mathbf{1}}

\newcommand{\Ma}{\mathcal{M}}
\newcommand{\mcal}{\mathcal}
\newcommand{\one}{\mathbbm{1}}
\newcommand{\Proj}{\operatorname{Proj}}
\newcommand{\de}{\mathbf{d}}
   \newcommand{\mfk}{\mathfrak}

\newcommand{\ed}{\mathscr{E}} 

\numberwithin{equation}{section}

\begin{document}
\title[Robustness of EKR theorems on permutations and perfect matchings.]{Robustness of Erd\H{o}s--Ko--Rado theorems on permutations and perfect matchings.} 
\date{}

\author[K. Gundrson]{Karen Gunderson}
\address{Department of Mathematics, University of Manitoba,	Winnipeg, 
	Manitoba R3T 2N2, Canada}\email{Karen.Gunderson@umanitoba.ca}
	\thanks{The ﬁrst author was supported by Natural Science and Engineering Research Council of Canada (grant RGPIN-2016-05949).
}
\author[K. Meagher]{ Karen Meagher}
\address{Department of Mathematics and Statistics, University of Regina,	Regina, 
	Saskatchewan S4S 0A2, Canada}\email{karen.meagher@uregina.ca}
	\thanks{The second author was supported by Natural Science and Engineering Research Council of Canada (grant RGPIN-03952-2018).}
\author[J. Morris]{Joy Morris}
\address{Department of Mathematics and Computer Science, University of Lethbridge,	Lethbridge, 
	Alberta T1K 3M4, Canada}\email{joy.morris@uleth.ca}
	\thanks{The third author was supported by the Natural Science and Engineering Research Council of Canada (grant RGPIN-2024-04013).}
\author[V.R.T Pantangi]{Venkata Raghu Tej Pantangi}
\address{Department of Mathematics and Statistics, University of Regina,	Regina, 
	Saskatchewan S4S 0A2, Canada}\email{pvrt1990@gmail.com}
\author[M.N. Shirazi]{Mahsa N. Shirazi}
\address{Department of Mathematics, University of Manitoba,	Winnipeg, 
	Manitoba R3T 2N2, Canada}\email{mahsa.nasrollahishirazi@umanitoba.ca}
\thanks{The authors are all indebted to the support of the Pacific Institute for Mathematical Sciences (PIMS), through the establishment of the Collaborative Research Group on Movement and Symmetry in Graphs which funded this work (report identifier PIMS-20240619-CRG36).}

\begin{abstract}
The Erd\H{o}s--Ko--Rado (EKR) theorem and its generalizations can be viewed as classifications of maximum independent sets in appropriately defined families of graphs, such as the Kneser graph $K(n,k)$. In this paper, we investigate the independence number of random spanning subraphs of two other families of graphs whose maximum independent sets satisfy an EKR-type characterization: the derangement graph on the set of permutations in $\sym{n}$ and the derangement graph  on the set $\Ma_{n}$ of perfect matchings in the complete graph $\mcal{K}_{2n}$. In both cases, we show there is a sharp threshold probability for the event that the independence number of a random spanning subgraph is equal to that of the original graph. As a useful tool to aid our computations, we obtain a Friedgut--Kalai--Naor (FKN) type theorem on sparse boolean functions whose domain is the vertex set of $\Ma_{n}$. In particular, we show that boolean functions whose Fourier transforms are highly concentrated on the first two irreducible modules in the $\sym{2n}$ module $\bb{C}[\Ma_{n}]$, is close to being the characteristic function of a union of maximum independent sets in the derangement graph on perfect matchings.       
\end{abstract}
\maketitle


\section{Introduction}

The Erd\H{o}s--Ko--Rado (EKR) theorem~\cite{EKR61} is a classical result in extremal set theory which deals with intersecting families of uniform sets. Given a positive integer $n$, we write $[n]$ for the set $\{1,2,\ldots, n\}$, and for any set $V$, we let $\binom{V}{k}$ denote the set of $k$-subsets of $V$. A family $\mcal{F}$ of $k$-subsets of $[n]$ is said to be \emph{intersecting} if $A \cap B \neq \emptyset$ for all $A, B \in \mcal{F}$. For any $x \in [n]$, the \emph{star centered at $x$}, or simply a \emph{star}, is the intersecting family 
\[
S_x = \left\{A \in \binom{[n] }{ k}\ :\ x \in A \right\}.
\]
The EKR theorem asserts that if $n \geq 2k$ and $\mcal{F} \subset \binom{[n] }{ k}$ is intersecting, then $|\mcal{F}|\leq \binom{n-1 }{ k-1}$ and, moreover, that provided $n>2k$, equality is achieved only when $\mcal{F}$ is a star. A collection of sets is \emph{canonically} intersecting if the intersection of all of the sets is nonempty. Since all of the $k$-sets in a star intersect in the same element, a star is canonically intersecting. Even in more general settings, ``star" is used to indicate a maximal canonically intersecting set.

The EKR theorem can be reformulated as a result on independent sets in  Kneser graphs. The Kneser graph $K(n,k)$ is a graph whose vertex set is $\binom{[n] }{ k}$ in which $A, B \in \binom{[n] }{ k}$ are adjacent if and only if $A\cap B = \emptyset$. Observe that $\mcal{F} \subset \binom{[n] }{ k}$ is intersecting if and only if $\mathcal{F}$ is an independent set in $K(n,k)$. Given a graph $G$, the independence number is the size of the maximum independent set in $G$. This size is denoted by $\alpha(G)$. For $n \geq 2k$, the EKR theorem states that $\alpha(K(n,k))=\binom{n-1 }{ k-1}$; furthermore, when $n>2k$, every maximum independent set in $K(n, k)$ is a star. The Hilton--Milner theorem~\cite{MR0219428} goes further and considers the largest independent sets in $K(n,k)$ that are not a subset of a star. Such a set can be no larger than $\binom{n-1 }{ k-1} - \binom{n-k-1 }{ k-1} + 1$; this shows that the stars are the largest intersecting sets by a significant margin.  

For any class of objects with a natural notion of intersection, one can ask for the largest collection of sets that are pairwise intersecting. This leads to many variations of the EKR theorem; we refer the reader to the book~\cite{MR3497070} for a survey of such results. In this paper, we consider versions of the EKR theorem for two types of objects: permutations and perfect matchings. For both of these objects a corresponding EKR theorem is known, and in both cases the largest intersecting collection is canonically intersecting (a star).

Two permutations $\sigma, \tau \in \sym{n}$ are said to be \emph{intersecting} if $\sigma(i)=\tau(i)$, for some $i \in [n]$. An \emph{intersecting set} in $\sym{n}$ is a subset $\calS \subset \sym{n}$ of pairwise intersecting permutations. Given any $i,j \in [n]$, the set $\calA_{i \to j}:=\{\sigma \in \sym{n} \ : \ \sigma(i)=j\}$ is a maximal canonically intersecting set (star), and has size $(n-1)!$. We call $\calA_{i\to j}$, the \emph{star centered at $i\to j$}.
 Inspired by the classical EKR theorem, Deza and Frankl~\cite{DF1977} asked for the size and structure of the largest intersecting subsets in $\sym{n}$. They proved that an intersecting set $\calS$ has size at most $(n-1)!$ and conjectured that every maximum intersecting set is necessarily a star. This conjecture was proved independently by Cameron and Ku~\cite{CK2003} and by Larose and Malvenuto~\cite{MR2061391}, and later another proof was given by Godsil and Meagher~\cite{MR3646689}. 
 
Like the standard EKR theorem, this result can be translated into a characterization of maximum independent sets in a graph, called the \emph{derangement graph}.
A \emph{derangement} is a fixed-point-free permutation, and we use $D(n)$ to denote the set of all derangements in $\sym{n}$. The \emph{derangement graph $\Gamma_{n}$ of $\sym{n}$}, is the Cayley graph $\cay(\sym{n},D(n))$. So the vertices of $\Gamma_n$ are the elements of $\sym{n}$ and two permutations $\sigma$ and $\tau$ are adjacent if $\sigma \tau^{-1} \in D(n)$. Observing that $\calA \subset \sym{n}$ is intersecting if and only if $\calA$ is an independent set in $\Gamma_{n}$, the EKR theorem for permutations is equivalent to the following result.
 
 \begin{theorem}
[Cameron and Ku~\cite{CK2003}, Deza and Frankl~\cite{DF1977}, Larose and Malvenuto~\cite{MR2061391}]
\label{thm:ekrsn}
The independence number of $\Gamma_{n}$ is $(n-1)!$ and every maximum independent set in $\Gamma_{n}$ is a star. 
\end{theorem}  

Now we explain the EKR variant on perfect matchings. An \emph{$n$-perfect matching} is a perfect matching on the complete graph $\mcal{K}_{2n}$ on $2n$ vertices. As a perfect matching is a set of edges, two $n$-perfect matchings are said to be \emph{intersecting} if they have a common edge. An intersecting set of perfect matchings is a set in which the perfect matchings are pairwise intersecting. Again, we ask for the size and structure of the largest sets of pairwise intersecting perfect matchings, and there is an obvious candidate. Given an edge $\{a,b\}$ of $\mcal{K}_{2n}$, the set $\calS_{\{a,b\} }$ of all perfect matchings containing $\{a,b\}$ is an intersecting set of size 
\[
(2n-3)!!= (2n-3) (2n-5) \cdots 1 = \dfrac{(2n-2)!}{2^{n-1}(n-1)!}.
\]
The set $\calS_{\{a,b\} }$ is called the \emph{star centered at $\{a,b\}$}, or simply a \emph{star}. A star is also known as a \emph{canonical intersecting set}. The Erd\H{o}s--Ko--Rado theorem for perfect matchings~\cite{MR3646689} states that the stars are the unique largest intersecting sets of perfect matchings.

As in the previous cases, this theorem can rephrased as a question about the maximum independent set in an appropriate graph. The \emph{perfect matching graph $\Ma_{n}$} is the graph whose vertices are all the perfect matchings on $\mcal{K}_{2n}$, with matchings $\mcal{P}$ and $\mcal{Q}$ being adjacent if and only if $\mcal{P} \cap \mcal{Q}= \emptyset$. It is easy to observe that a family of $n$-perfect matchings is intersecting if and only if it forms an independent set in $\Ma_{n}$. The set $\calS_{\{a,b\} }$ is an independent set in $\Ma_{n}$ of size $(2n-3)!!$.
 
\begin{theorem}[Godsil--Meagher~\cite{MR3646689}]\label{thm:ekrpm}
Given any positive integer $n$,  $\alpha(\Ma_{n})=(2n-3)!!$ and every maximum independent set in $\Ma_{n}$ is a star.
\end{theorem} 

A relatively recent research direction in this field involved considering other measures of how robust the standard EKR theorem is by considering the largest intersecting sets of random subgraphs of $K(n,k)$. There are two natural approaches given by  either deleting vertices or edges from the Kneser graph. 
 Balogh, Bohman and Mubayi~\cite{balogh2009erdHos} considered taking random induced subgraphs of the Kneser graph, with each vertex included independently with probability $p$. They gave results about values of $p$ for which the maximum independent sets in random subgraphs were subsets of a star; that is the vertices in these independent sets were all of the retained vertices that corresponded to sets that contained a fixed point. Refinements of this result for different ranges of $k$ (in terms of $n$) are given by Hamm and Khan~\cite{HK19A, HK19B}; Gauy, H{\`a}n, and Oliveira~\cite{GHO17}; and Balogh, Das, Delcourt, Liu, and Sharifzadeh~\cite{BDDLS15}.
 
 The other approach is to delete edges of a Kneser graph at random to obtain a random spanning subgraph. Given a graph $G$ and a probability $p\in [0,1]$, let $G_{p}$ denote the random model for spanning subgraphs of $G$ with each edge included independently at random with probability $p$. In the case $G$ is the Kneser graph $K(n,k)$, the random graph model $K_{p}(n,k)$ was first considered by Bogolyubski{\u{\i}}, Gusev, Pyad{\"{e}}rkin, and Ra\u{\i}gorodski\u{\i}~\cite{BGPR14, BGPR15}, who showed that $\alpha\left(K_{1/2}(n,k) \right)= \alpha\left(K(n,k)\right)(1+o(1))$. Bollobas, Narayanan and Ra\u{\i}gorodski\u{\i}~\cite{BNR2016} investigated threshold functions for the event $\alpha\left(K_{p}(n,k) \right)= \alpha\left(K(n,k)\right)$, showing that for $k=o(n^{1/3})$, there is a threshold probability $p_{c}=p_{c}(n,k)$ so that 
 (i) if $p>>p_{c}$, then with high probability $\alpha\left(K_{p}(n,k) \right)= \alpha\left(K(n,k)\right)$; and 
 (ii) if $p<<p_{c}$, then with high probability $\alpha\left(K_{p}(n,k) \right)> \alpha\left(K(n,k)\right)$. 
 Their results were extended to a larger range of $k$ in terms of $n$ by Balogh, Bollob{\`a}s, Narayanan~\cite{BBN2015}, Devlin and Khan~\cite{DK2016}, Das and Tran~\cite{DT2015}, culminating in the work of Balogh, Kruger, and Luo~\cite{BKL2023}---who extended the range to $k= (n-1)/2$. Note that in the case $k=n/2$, the graph $K(2k,k)$ is a perfect matching, so the event $\alpha\left(K_{p}(n,k) \right)= \alpha\left(K(n,k)\right)$ is precisely the event in which no edges were deleted. We summarize their results in the following result.

\begin{theorem}[\cite{BBN2015, BKL2023, BNR2016, DK2016, DT2015}]\label{thm:randomekrset}
Let $k=k(n)$ and set 
\[
p_{c}(n,k):=\begin{cases}
\dfrac{3}{4} &\text{if $n=2k+1$,} \\ 
\frac{ \ln \left(n \ \binom{n-1 }{ k} \right)}{\binom{n-k-1 }{ k-1}} &\text{if $n>2k+1$.} 
\end{cases}
\]
Then, for any $\varepsilon >0$ as $n\to \infty$,
\[
\bb{P} \left[ \alpha(K_{p}(n,k)) =\binom{n-1 }{ k-1} \right] \to 
\begin{cases} 
1 & \text{if $p \geq (1+\varepsilon)p_{c}(n,k)$,} \\ 
0 & \text{if $p \leq (1-\varepsilon)p_{c}(n,k)$.}   
\end{cases} 
\]
Furthermore, when $p\geq (1+\varepsilon)p_{c}$, then with high probability every maximum independent set in $K_{p}(n,k)$ is a star.
\end{theorem}    


We have seen that Theorem~\ref{thm:ekrsn} and Theorem~\ref{thm:ekrpm} are equivalent to determining the maximum independent sets in $\Gamma_{n}$ and $\Ma_{n}$, respectively. In this paper, we consider random analogues of these two graphs. 

Our first results are analogous to Theorem~\ref{thm:randomekrset}, but in the context of $\Gamma_n$ and $\Ma_n$ rather than $K(n,k)$. More precisely, we show that using natural random analogues for each family of graphs, we are able to determine threshold probabilities. These threshold probabilities have the property that if the probability used in the random analogue exceeds the threshold, then with high probability the independence number of the random analogue will have the same independence number as the original graph, and if the probability is below the threshold then with high probability the random analogue will have a lower independence number than the original. 

We begin with $\Gamma_n$. Given $p \in [0,1]$, by $\Gamma_{n,p}$, denote the random spanning subgraph of $\Gamma_{n}$ where edges are included independently with probability $p$. We find threshold probabilities for the event $\alpha(\Gamma_{n,p})=(n-1)!$, and thus prove a random analogue of Theorem~\ref{thm:ekrsn} just as Theorem~\ref{thm:randomekrset} gives a random analogue of the EKR theorem.

\begin{theorem}\label{thm:randomekrsn}
Let $d_{n}$ be the number of derangements in $\sym{n}$. Let 
\[
p_{c}(n):=\dfrac{(n-1)\ln\left(n! \ n(n-1)\right)}{d_{n}}.
\]
Then, for any $\varepsilon>0$, as $n\to \infty$,
\[
\bb{P} \left[ \alpha(\Gamma_{n,p}) =(n-1)! \right] \to 
\begin{cases} 
1 & \text{if $p \geq (1+\varepsilon)p_{c}(n)$}, \\ 
0 & \text{if $p \leq (1-\varepsilon)p_{c}(n)$.}   
\end{cases} 
\]
Furthermore, when $p\geq (1+\varepsilon)p_{c}(n)$, with high probability, every maximum independent set in $\Gamma_{n,p}$ is a star.
\end{theorem}

Next we look at $\Ma_{n}$. Given $p \in [0,1]$, by $\Ma_{n, p}$, we denote the random spanning subgraph of $\Ma_{n}$ where edges are included independently with probability $p$. We find threshold probabilities for the event 
$\alpha(\Ma_{n,p})=(2n-3)!!$, giving a random analogue of Theorem~\ref{thm:ekrpm}.
 
\begin{theorem}\label{thm:randomekrpm}
Let $\varepsilon>0$ be a fixed constant and $\de_{n}$ be the degree of the perfect matching graph $\Ma_{n}$. Let 
\[
p_{c}(n):=\dfrac{(2n-2)\ln(n(2n-2)(2n-1)!!)}{\de_{n}}.
\]
Then, for any $\varepsilon >0$, as $n\to \infty$,
\[
\bb{P} \left[ \alpha(\Ma_{n,p}) =(2n-3)!! \right] \to 
  \begin{cases} 
   1 & \text{if $p \geq (1+\varepsilon)p_{c}(n)$}, \\ 
   0 & \text{if $p \leq (1-\varepsilon)p_{c}(n)$.}  
   \end{cases} 
\]
Furthermore, when $p\geq (1+\varepsilon)p_{c}(n)$, with high probability, every maximum independent set in $\Ma_{n,p}$ is a star.
\end{theorem}

Let $G$ be a graph and $p\in [0,1]$ a probability.  Let $G_p$ denote the random spanning subgraph of $G$ where edges are included independently with probability $p$. To compute the independence number of $G_{p}$, it is useful to obtain bounds on edge densities of induced subgraphs in $G$. Given a subset $S$ of vertices in $G$, by $\ed(S)$, we denote the number of edges in the subgraph of $G$ induced by $S$; we also say that $S$ \emph{spans} $\ed(S)$ edges. The set $S$ will only be independent in $G_{p}$ if all the edges spanned by $S$ are removed, so the probability of $S$ being independent in $G_{p}$ is $(1-p)^{\ed(S)}$.  In the case $\ed(S)$ is ``large'', the probability of $S$ being independent in a random subgraph is ``low''. To prove our results, we require structural information about sets of vertices that induce sparse graphs. Clearly, stars induce empty graphs and are extremal in the sense of inducing sparse subgraphs. 
 
Using a variant---Proposition~\ref{prop:edge-count-norm}---of the expander mixing lemma, we can use the eigenvalues of a $k$-regular graph to produce a lower bound on the number of edges spanned by a set of vertices in the graph.

\begin{proposition}
Let $G$ be a $k$-regular graph on $n$ vertices and let $\tau$ be the smallest eigenvalue of $G$. For any 
subset of vertices $S$, 
\begin{equation}\label{eq:isoperineq}
 \ed(S) \geq  \frac{|S|^2}{2n}(k-\tau) + \dfrac{|S|\tau}{2}.
\end{equation} 
\end{proposition}

There is a simple equivalence between sets of vertices in $G$, and boolean functions on the vertices of $G$ via an indicator function; this is simply the function that returns 1 for vertices in $S$ and 0 otherwise.

By Proposition~\ref{prop:edge-count-norm}, $S$ satisfies \eqref{eq:isoperineq} with equality if and only if its indicator function $\ind_{S}$ is a boolean function in the sum of the eigenspaces associated with $k$ and $\tau$; call this space $U_{G}$ . When $G$ is either $\Gamma_{n}$ or $\Ma_{n}$, the space $U_{G}$ is linearly spanned by indicator functions of stars (see Lemma~\ref{lem:specdergraph} and Lemma~\ref{lem:specpmgraph}). In other words, in these cases, $U_{G}$ is a set of degree $1$ polynomials in indicator functions of stars. Such functions are known as degree $1$ functions in the underlying association scheme. The definition of boolean degree $1$ functions has been extended to some other classical association schemes~\cite{FI2019}. 
 
Boolean degree $1$ functions in association schemes are extremal in the sense that the associated sets satisfy \eqref{eq:isoperineq} with equality. It is of interest to find the structure of \emph{almost-extremal} sets, that is, sets whose indicator functions are ``close'' to satisfying \eqref{eq:isoperineq} with equality. By Proposition~\ref{prop:edge-count-norm}, indicator functions of almost-extremal sets must be ``close'' (in euclidean distance) to $U_{G}$.   

When $G$ is a hypercube, the classical Friedgut--Kalai--Naor (FKN) theorem~\cite{FKN2002} shows that a boolean function which is ``close'' to $U_{G}$, must be in fact close to a degree $1$ boolean function. The FKN theorem has been generalized to other discrete domains: for example to graph products in~\cite{AFS2004}, to the Kneser Graph~\cite{Filmus16}, and to $\Gamma_{n}$~\cite{EFF2015, EFF2015b, Filmus21}. We extend the techniques of~\cite{EFF2015} to find an FKN result in the domain of perfect matchings. We first state the main result of~\cite{EFF2015} which is an FKN result for $\Gamma_n$. We use $\Delta$ to denote the symmetric difference of two sets, and $\mrm{round}(c)$ for the integer closest to $c$ (we use the convention of rounding up if two integers are equally close, but this is not important to our results).

\begin{theorem}[{Ellis, Filmus, Friedgut~\cite[Theorem 1]{EFF2015}}]\label{thm:EFF-stability}
Let $U$ be the linear span of indicator functions of stars in $\Gamma_{n}$, and let $A \subseteq \sym{n}$ be any set of size $ c(n-1)!$, where $c\leq \frac{n}{2}$.  There exist absolute constants $C_0$ and $\varepsilon_0$ such that whenever $\varepsilon <\varepsilon_0$ the following is satisfied.
If $\| \ind_{A} - \operatorname{Proj}_U(\ind_{A})\|^2 \leq \varepsilon/n$, then there is a set $B$ that is a  union of $\mrm{round}(c)$ stars in $\Gamma_{n}$, such that
\[
| A \Delta B |  \leq   C_0 c^{2}(\varepsilon^{1/2} + 1/n) (n-1)!. 
\]
Moreover $|c-\mrm{round}(c)| \leq C_{0} c^{2}\left(\sqrt{\varepsilon} + \frac{1}{n} \right)$.  
\end{theorem}      

We obtain the following FKN type result in $\Ma_{n}$.

\begin{theorem}\label{thm:EFFPM}
Let $U$ be the linear span of indicator functions of stars in $\Ma_{n}$, and let 
$A\subseteq \mathcal{M}_{n}$ be any set of size $c(2n-3)!!$, 
where $c \leq \frac{2n-1}{2}$. There exist absolute constants $C_{0}$ and $\varepsilon_{0}$ such that whenever $\varepsilon \le \varepsilon_0$ the following is satisfied. If $\|\ind_{A}-\operatorname{Proj}_{U}(\ind_{A})\|^2\leq \varepsilon\frac{ c}{2n-1}$,  then there is a set $B$ that is a union of $\mrm{round}(c)$ stars in $\Ma_{n}$ such that  
\[
|A \Delta B| \leq C_{0} c^{2} (2n-3)!! \left( \sqrt{\varepsilon} + \frac{1}{2n-1} \right).
\]
Moreover $|c-\mrm{round}(c)| \leq C_{0} c^{2}\left(\sqrt{\varepsilon} + \frac{1}{2n-1} \right)$.
\end{theorem}

As an application of their main result (Theorem~\ref{thm:EFF-stability}), the authors of~\cite{EFF2015} proved some stability results on intersecting sets of permutations in $\sym{n}$. We will prove analogues of these results for independent sets in $\Ma_{n}$. Our method of proof was originally used in~\cite[\S4]{EFF2015} to prove a conjecture (c.f~\cite[Conjecture 2]{EFF2015}) by Cameron and Ku. The following result was originally proved in~\cite{lindzey2020stability}, using different methods. One application of Theorem~\ref{thm:EFFPM} is to provide an alternative proof.

\begin{theorem}\label{thm:pmweakstability}
There exists $\delta$ such that for all $n$,  if $A$ is an independent set of size at least $(1-\delta) (2n-3)!!$ in $\Ma_{n}$, then $A$ is contained in a star of $\Ma_{n}$.  
\end{theorem}

\section{Notation, Background, and Preliminary results.}

In the following sections we will give the needed background and preliminary results.

\subsection{The graphs $\Gamma_{n}$ and $\Ma_{n}$.}

In this section, we define some notation and provide some previously-known results on $\Gamma_{n}$ and $\Ma_{n}$. For any set $V$ (typically this will be the set of vertices of a graph) with $S \subset V$, the \emph{indicator function} of $S$ is the length-$|V|$ vector defined by
 \[
 \ind_{S}(v) := 
 \begin{cases}
 1 &\text{if $v \in S$,} \\ 
 0 &\text{otherwise.} 
 \end{cases} 
 \]   
For a simple graph $G=(V,E)$ we denote the space of $\bb{R}$-valued functions with domain $V$ by $\bb{R}[V]$.
We use $A_{G}$ to denote the adjacency matrix of $G$, and for an eigenvalue $\omega$ of $A_{G}$, the corresponding eigenspace in $\bb{R}[V]$ is denoted by $V_{\omega}$.

Consider the graph $\Gamma_{n}$. This is the {normal Cayley graph} 
$\cay( \sym{n}, D(n))$. The phrase ``normal Cayley graph" has two very different meanings in the literature. It is often used to mean that the regular subgroup is normal in the automorphism group of the graph.  However, when the adjacency matrix and eigenvalues of a graph are being studied, it typically carries a different meaning. In this context, \emph{normal} means that the connection set, the set of all derangements, is closed under conjugation. This is the way in which we use the term, and it yields significant information about the graph. 

Any Cayley graph, so in particular $\Gamma_n$, is a regular graph whose valency is the cardinality of the connection set. In the case of $\Gamma_n$, this means that it is a
$d_{n}$-regular graph, where $d_{n}$ is the number of derangements in $\sym{n}$. 
The number of derangements is well-known, see, for example,~\cite[Example~2.2.1]{stanley2011enumerative} and the following can be determined by counting using inclusion and exclusion:  
\begin{align}\label{eq:numderangements}
d_{n} = n! \sum_{i=0}^{n} (-1)^i \frac{1}{i!} = \left\lfloor\dfrac{n!}{e}+\dfrac{1}{2}\right\rfloor= n!\left(\dfrac{1}{e} + o(1) \right).
\end{align}  

The normality of $\Gamma_n$ also shows that the eigenvalues of $\Gamma_{n}$ can be found using the irreducible representations of $\sym{n}$ (see~\cite[Chapter 14]{MR3497070} for details). 
In short, the space $\bb{R}[\sym{n}]$ is isomorphic to the real regular representation of $\sym{n}$, so it can be decomposed into irreducible submodules, and there is a submodule for every irreducible representation of $\sym{n}$. It is well known that there is a canonical one-to-one correspondence between irreducible representations of $\sym{n}$ and the integer partitions of $n$. This correspondence can be found in many classical books on character theory, such as~\cite[Chapter 4]{FH91}. Given an integer partition $\lambda$ of $n$, by $\bf{\chi_\lambda}$ we denote the irreducible character of $\sym{n}$ associated with $\lambda$. Further, $U_{\lambda}$ will denote the isotypic component in $\bb{R}[\sym{n}]$ corresponding to $\bf{\chi_{\lambda}}$. 
Each of these components is a subspace of an eigenspace for $\Gamma_n$, so there is an eigenvalue associated to each $U_\lambda$, and this eigenvalue can be calculated by the formula from~\cite{MR0546860, MR0626813}
\[
\xi_\lambda = \frac{1}{\chi_\lambda (1) } \sum_{x \in D(n)} \chi_\lambda(x). 
\]
 With this formula, it is straightforward to calculate that the eigenvalues corresponding to the representations $\chi_{(n)}$ and $\chi_{(n-1,1)}$ are $d_n$ and $\frac{-d_n}{n-1}$, respectively. 
Further, the module $U_{(n)}$ is spanned by the all-ones vector (this is the eigenspace corresponding to the valency) and it is also not difficult to find a spanning set for the space $U = U_{(n)} + U_{(n-1,1)}$.  More information about the other eigenvalues of $\Gamma_{n}$ is also known, see for instance~\cite[pages 176--177]{ellis2012proof}. We summarize the results we need in the following lemma.
  
\begin{lemma}[see references and discussion above]\label{lem:specdergraph} 
For every positive integer $n$, $\Gamma_n$ has the following properties:
\begin{enumerate}
\item it is a $d_{n}$-regular graph, and $d_n$ is the largest eigenvalue; 
\item it has $\dfrac{-d_{n}}{n-1}$ as its least eigenvalue, and all other eigenvalues are $O((n-2)!)$;
\item $U_{(n)} =\mrm{Span} \left( \{ \ind_{V} \} \right)$ is the $d_{n}$-eigenspace;
\item $U_{(n-1,1)}$ is the $\dfrac{-d_{n}}{n-1}$-eigenspace;
\item $U_{(n)} +U_{(n-1,1)}=\mrm{Span} \left( \left\{\ind_{\calS_{i\to j}} \ : \ i,j \in [n] \right\} \right)$.
\end{enumerate}
\end{lemma}  

Next we turn to the perfect matching graph, $\Ma_n$. The vertex set of this graph is the set of all perfect matchings of $\mathcal{K}_{2n}$. Since the automorphism group of $\mathcal K_{2n}$ is $\sym{2n}$ which acts transitively on the perfect matchings of $\mathcal K_{2n}$, the graph $\Ma_n$ is vertex-transitive and therefore in particular is regular. We denote the degree by $\de_{n}$.
Similar to the derangement graph, the value of $\de_n$ can be found by counting, using inclusion and exclusion (see for instance~\cite[\S2]{lindzey2020stability})
\begin{equation}\label{eq:numderpm}
\de_{n}= \sum_{i=0}^{n} (-1)^i \binom{n }{ i} (2n-2i-1)!!
      = (2n-1)!! \left(\dfrac{1}{\sqrt{e}} + o(1)\right).
\end{equation}

The natural action of $\sym{2n}$ on the perfect matchings gives a permutation representation of $\sym{2n}$. The decomposition of this representation is known to be the sum of the irreducible representations of $\sym{2n}$ that correspond to partitions in which all the parts are even (for a proof of this, see~\cite{MR627512}). 
Given an integer partition $\lambda$ of $2n$, in which all the parts are even, let $\mfk{U}_{\lambda}$ be the $\chi_{\lambda}$-isotypic component in $\bb{R}[V(\Ma_n)]$. This space is contained in an eigenspace of the adjacency matrix of $\Ma_n$, and there is an eigenvalue that belongs to it,   see~\cite[\S15.2]{MR3646689} for details. There is a formula to calculate these eigenvalues but it is more difficult than the equation for the eigenvalues of the derangement graph. These eigenvalues have been studied and some recursive formulas for them have been found, see~\cite{MR4510077, MR1285207, MahsaThesis, MR4113598}. The following result summarizes~\cite[Theorem~7.2, Lemma~8.1]{MR3646689}, and~\cite[Lemma 15]{lindzey2020stability}. 
Similar to the case for permutations, it is not difficult to find a spanning set for the space 
$\mfk{U} = \mfk{U}_{(2n)} + \mfk{U}_{(2n-2,2)}$. Again we summarize the results we need in a lemma.

\begin{lemma}\label{lem:specpmgraph}[Godsil-Meagher~\cite{MR3646689}, Lindzey~\cite{lindzey2020stability}]
For every positive integer $n$, $\Ma_{n}$ has the following properties:
\begin{enumerate}
\item it is a $\de_{n}$-regular graph and $\de_n$ is the largest eigenvalue;
\item $\dfrac{-\de_{n}}{2n-2}$ is its least eigenvalue, and all other eigenvalues are $O((2n-5)!!)$;
\item $\mfk{U}_{(2n)} =\mrm{Span} \left( \{ \ind_{V(\Ma_n)} \} \right)$ is the $\de_{n}$-eigenspace;
\item $\mfk{U}_{(2n-2,2)}$ is the $\dfrac{-\de_{n}}{2n-2}$-eigenspace.
\item $\mfk{U}_{(2n)} +\mfk{U}_{(2n-2,2)} = \mrm{Span} \left(\left\{ \ind_{\calS_{\{a,b\}}} \ : \ \{a,b\} \in E(\mcal{K}_{2n}) \right\} \right)$.
\end{enumerate}
\end{lemma}

For both the derangement graph and the perfect matchings graph we know the largest eigenvalue and the smallest eigenvalues, so the well-known \emph{Delsarte--Hoffman ratio bound} can be applied. In both cases the bound holds with equality and can be used to prove a version of the EKR theorem for these objects. The following form of the bound is found in~\cite[Corollary~2.4.3]{MR3497070} and holds for $\Gamma_n$ and $\Ma_n$.

\begin{theorem}[Delsarte--Hoffman ratio bound] \label{thm:ratiobound}
Let $G$ be a $k$-regular graph on $v$ vertices with $\tau$ as its least eigenvalue. Then 
\begin{enumerate}
\item $\alpha(G) \leq \dfrac{v}{1- {k}/{\tau}}$; and
\item if $S$ is an independent set such that 
\[
|S| = \dfrac{v}{1-k/\tau},
\] 
then every vertex outside $S$ has exactly $-\tau$ neighbours in $S$.
\end{enumerate} 
\end{theorem}

\subsection{Isoperimetry Results}

Recall that for a given graph $G$ and a probability $p \in [0,1]$, the graph $G_{p}$ is the random spanning subgraph of $G$ where each edge is included independently with probability $p$.  Further, for a set of vertices $S$ in $G$, the probability of $S$ being independent in $G_{p}$ is $(1-p)^{\ed(S)}$, where $\ed(S)$ denotes the number of edges in the subgraph of $G$ induced by $S$.


It is useful to get some bounds on $\ed(S)$ for sets $S$ of vertices in $\Gamma_{n}$ and $\Ma_{n}$. Bounds on the size of $\ed(S)$ are known as \emph{isoperimetry} results, and will be key in considering independent sets in random subgraphs. 
Our first isoperimetry result is a simple bound stemming from the Delsarte--Hoffman ratio bound. 

\begin{cor}\label{cor:edge-count-ratio}
Let $G=(V,E)$ be a $k$-regular graph with least eigenvalue $\tau$. Assume that $\alpha(G)= \dfrac{v}{1- {k}/{\tau}}$, and let $S$ be an independent set of size $\alpha(G)$. If $A\subset S$ and $B\subset V\setminus S$, then taking $T= (S\setminus A) \cup B$ yields 
\[
\ed(T) \geq |B|(-\tau-|A|).
\]
\end{cor}
\begin{proof}
By Theorem~\ref{thm:ratiobound}, each $b \in B$ has exactly $-\tau$ neighbours in $S$. Therefore each $b \in B$ has at least $-\tau-|A|$ neighbours in $S\setminus A$, and the result follows.
\end{proof}

Combining Theorem~\ref{thm:ekrsn} with Lemma~\ref{lem:specdergraph} and observing that $\Gamma_n$ has $n!$ vertices, we see that $\alpha(\Gamma_n)$ meets the Delsarte--Hoffman ratio bound with equality. Similarly, combining Theorem~\ref{thm:ekrpm} with Lemma~\ref{lem:specpmgraph} and observing that $\Ma_n$ has $(2n-1)!!$ vertices, we see that $\alpha(\Ma_n)$ also meets the Delsarte--Hoffman ratio bound with equality. Thus, both graphs $\Gamma_{n}$ and $\Ma_{n}$ satisfy the premise of the above result. We also note that the above result is vacuous in the case $|A| \geq -\tau$. In such cases, we use a different isoperimetry result which we now describe. Adapting the methods of Das and Tran~\cite[Theorem 1.6]{DT2015}, we obtain the following ``spectral'' lower bound on edge density of spanning subgraphs of a regular graph. This bound uses a norm for $f \in \bb{R}[V]$, defined by
\[
\| f \|= \left(\frac{1}{|V|} \sum \limits_{v \in V} f(v)^{2}\right)^{1/2}.
\]

\begin{proposition}\label{prop:edge-count-norm}
Let $G=(V, E)$ be a connected $k$-regular graph with smallest eigenvalue $\tau$ and second smallest eigenvalue $\mu$.  
For every eigenvalue $\xi$, let $V_{\xi}$ be the corresponding real eigenspace and let $U = V_k \oplus V_{\tau}$.  

For any $S \subset V$ 
\[
\frac{2\ed(S)}{|V|} \geq  (k-\tau)  \frac{|S|^2}{|V|^2}+   \tau  \frac{|S|}{|V|} - (\tau - \mu) \ \| \ind_S - \operatorname{Proj}_U( \ind_S )\|^2.
\]
\end{proposition}

\begin{proof}
Define $f = \ind_S$ and $f_0 := \operatorname{Proj}_{V_k}(f)$. Since $V_k$ is spanned by the all ones vector, $f_0 = \frac{|S|}{|V|}\ind_{V}$ and $\|f_0\|^2 = \frac{|S|^2}{|V|^2}$. Next define $f_1 = \operatorname{Proj}_{V_\tau}(f)$ and $f_2 = f - f_0 - f_1$, so that $f_2 = f - \operatorname{Proj}_U(f)$.  
These definitions ensure that  $f_0$, $f_1$, $f_2$ are pairwise orthogonal.
Furthermore, note that since $f = \ind_S$, it is easy to see that $ \| f \|^2 = \frac{|S|}{|V|}$.

If $A$ is the adjacency matrix for $G$, then $f^T A f = 2\ed(S)$, as the subgraph in $G$ induced by $S$ contains exactly $\ed(S)$ edges. We can also express $f^T A f$ by expanding $f$ into its pairwise orthogonal components $f_0$, $f_1$, $f_2$ to get 
\begin{align*}
f^T A f  &=       f_0^T A f_0 + f_1^T A f_1 + f_2^T A f_2\\
	   &=       k \|f_0\|^2 |V| + \tau \|f_1\|^2 |V| + f_2^T A f_2\\\
	   &\geq  k \|f_0\|^2 |V| + \tau (\|f\|^2 - \|f_0\|^2 - \|f_2\|^2)  |V| + \mu \|f_2\|^2 |V| \\
	   &=  |V| \left(  (k- \tau) \|f_0\|^2 + \tau \|f\|^2  - (\tau - \mu) \|f_2\|^2 \right).
\end{align*} 
Using the fact that $f^T A f = 2\ed(S)$, we have
\[
\frac{2\ed(S)}{|V|} 
	 \geq (k- \tau) \|f_0\|^2 + \tau \|f\|^2  - (\tau - \mu) \|f_2\|^2
	 = (k-\tau) \frac{|S|^2}{|V|^2} + \tau \frac{|S|}{|V|} - (\tau - \mu) \|f_2\|^2,
\]
which completes the proof.
\end{proof}

Using Theorem~\ref{thm:EFF-stability} and the above bound, we derive another isoperimetry result for sets in $\Gamma_{n}$. Prior to stating this result, we state the following useful corollary of Theorem~\ref{thm:EFF-stability}. This result is the contrapositive of Theorem~\ref{thm:EFF-stability} in the case $c=1$. 

\begin{cor}\label{cor:EFFSn}
Let $U$ be the linear span of indicator functions of stars in $\Gamma_{n}$.
There exist absolute constants $C_{0}$ and $\varepsilon_{0}$ so that for any $A\subset \sym{n}$ with size $(n-1)!$, the following holds.

If there is an $\varepsilon \leq \varepsilon_{0}$ such that  
\[
|A \Delta \calS_x| \geq C_{0}  (n-1)! \left( \sqrt{\varepsilon} + \frac{1}{n} \right),
\] 
for every star $\calS_x$,
then 
\[
\|\ind_{A}-\operatorname{Proj}_{U}(\ind_{A})\|^2\geq \frac{\varepsilon}{n}.
\]
\end{cor}

We now use Proposition~\ref{prop:edge-count-norm} and Corollary~\ref{cor:EFFSn} to find a lower bound for the number of edges induced by a set of vertices in $\Gamma_n$, in the case that the set does not have a large intersection with any of the stars.

\begin{theorem}\label{thm:Sn-edge-lb}
There is an absolute constant $\kappa > 0$ so that for every $\delta > 0$, there is an $n_{\delta} \in \bb{N}$ such that the following is true for all $n>n_{\delta}$ and $i$ with $\delta \dfrac{d_{n}}{n-1} \leq i \leq (n-1)!$.

If $A \subseteq \sym{n}$ with $|A| = (n-1)!$ and $|A \setminus \calS_x | \geq i$ for every star $\calS_x$, then 
\[
\ed(A) \geq \kappa i^{2}.
\]
\end{theorem}

\begin{proof}
As we have seen in Lemma~\ref{lem:specdergraph}, the smallest eigenvalue for the adjacency matrix of $\Gamma_{n}$ is $\tau = -\frac{d_{n}}{n-1} = \Theta((n-1)!)$, and the absolute value of its second smallest eigenvalue is $|\mu| = O((n-2)!)$.  Thus, for $n$ large enough,
\begin{equation}\label{eq:useful1}
-(\tau - \mu) \geq \frac{-\tau}{2} = \frac{d_{n}}{2(n-1)}.
\end{equation}
Applying Proposition~\ref{prop:edge-count-norm} in this case shows that for any set $A\subset \sym{n}$ of size $(n-1)!$, we have 
\begin{equation} \label{ineq:edAProjBound}
\begin{aligned}
\ed(A) &\geq \frac{n!}{2}\left(  \left( d_{n} + \frac{d_{n}}{n-1}\right)  \frac{(n-1)!^2}{n!^2} - \frac{d_{n}}{n-1}\frac{(n-1)!}{n!} - (\tau - \mu) \ \|\ind_{A} - \operatorname{Proj}_U(\ind_{A})\|^2  \right)\\ 
 &= \frac{n!}{2} \left(d_{n} \left( \frac{1}{n^2}\left(1 + \frac{1}{n-1} \right) - \frac{1}{n(n-1)}  \right) - (\tau - \mu)
    \ \|\ind_{A} - \operatorname{Proj}_U(\ind_{A})\|^2 \right)\\ 
 &= \frac{n!}{2}(-\tau+\mu)
   \  \|\ind_{A} - \operatorname{Proj}_U(\ind_{A})\|^2 \\ 
 &\geq \frac{n!}{2} \ \frac{d_{n}}{2(n-1)}
  \   \|\ind_{A} - \operatorname{Proj}_U(\ind_{A})\|^2 \qquad \text{(by \eqref{eq:useful1})}  \\
 &=\frac{n! |\tau| }{4} \ \|\ind_{A} - \operatorname{Proj}_U(\ind_{A})\|^2.
\end{aligned}
\end{equation}

For ease of notation, we set $M=|\tau| = \frac{d_n}{n-1}$. Then let $\delta>0$ and $i$ be such that $\delta M \leq i \leq (n-1)!$.  Consider a set $A$ such that $|A\setminus \calS_x| \geq i$, for every star $\calS_x$.
Since $|A|=(n-1)!=|\calS_x|$, we have
\[
| A \Delta \calS_x | = 2 | A \setminus \calS_x | \geq 2i.
\]

With $C_{0}$ and $\varepsilon_{0}$ as in Theorem~\ref{thm:EFF-stability}, let $C_1>0$ be large enough so that
\begin{equation}\label{eq:C1}
2C_{0} \leq C_1, \qquad  \frac{4}{C_1^2} \leq \varepsilon_0.    
\end{equation}

From \eqref{eq:numderangements}, we have $\lim\limits_{n \to \infty} \dfrac{M}{(n-1)!} =\dfrac{1}{e}$, so for $n$ being sufficiently large, we have $\dfrac{M}{(n-1)!} \geq \dfrac{1}{2e}$. Therefore we can pick $n_{\delta}$ so that for all $n \geq n_{\delta}$ the following two inequalities hold
\begin{equation}\label{eq:nd}
  \dfrac{1}{n} \leq  \dfrac{\delta}{eC_{1}},  \qquad \dfrac{1}{e} \leq  \dfrac{2M}{(n-1)!}.
\end{equation}

Next, set $\varepsilon = \left(\frac{2i}{C_1(n-1)!}\right)^2$. Since $i \leq (n-1)!$, by \eqref{eq:C1}, we have 
\[
\varepsilon=  \left(\frac{2i}{C_1(n-1)!}\right)^2 \leq \frac{4}{C_{1}^{2}} \leq \varepsilon_{0}. 
\]
Applying the two inequalities in~\eqref{eq:nd} followed by the fact $i \geq \delta M$, we have for all $n \geq n_{\delta}$ that 
\[
\dfrac{1}{n} \leq  \dfrac{\delta}{ e C_{1} }
                   \leq  \dfrac{2 M}{C_{1}(n-1)!} 
                   \leq   \dfrac{2 i}{C_{1}(n-1)!} 
                   = \varepsilon^{1/2},
\]
(the final equality is simply the definition of $\varepsilon$).

Assume $n \geq n_{\delta}$. Applying the equality and the two inequalities $2i=C_{1} (n-1)! \varepsilon^{1/2}$, $2C_{0} \leq C_{1}$, and $\dfrac{1}{n} \leq \varepsilon^{1/2}$, shows for any star $\calS_x$
\[
 |A \Delta \calS_x | 
 \geq 2i 
 = C_{1} (n-1)! \varepsilon^{1/2} 
\geq   2C_{0}  (n-1)! \varepsilon^{1/2} \]
\[
 =  C_{0}  (n-1)! ( \varepsilon^{1/2} +\varepsilon^{1/2})  
= C_{0} (n-1)! (\varepsilon^{1/2} + 1/n).
\]
Thus, by Corollary~\ref{cor:EFFSn}, 
\[
||\ind_{A} - \operatorname{Proj}_U(\ind_{A})||^2 \geq \frac{\varepsilon}{n} = \frac{4 i^2}{C_1^2 n! (n-1)!}.
\]
Using the above inequality in~\eqref{ineq:edAProjBound} yields
\[
\ed(A) \geq \frac{n! M }{4} ||\ind_{A} - \operatorname{Proj}_U(\ind_{A})||^2 
           \geq  \frac{n! M }{4 } \ \frac{4 i^2}{C_1^2 n! (n-1)!} 
            =  \frac{M i^2}{ C_1^2 (n-1)!}.
\]
By the second inequality in~\eqref{eq:nd}, $\dfrac{1}{e} \leq  \dfrac{2M}{(n-1)!}$. It follows that
\[
\ed(A) \geq \frac{M i^2}{ C_1^2 (n-1)!} \geq i^2 \frac{1}{ 2e  C_1^2}.
\]
Thus $A$ contains at least $\kappa i^{2}$ edges, where $\kappa=  1/ ( 2 e C_{1}^{2})$. 
\end{proof}

Using a similar argument, we derive an analogous isoperimetry result for sets of vertices in $\Ma_{n}$.

\begin{theorem}\label{thm:pmedlb}
 There is an absolute constant $\kappa > 0$ so that for every $\delta > 0$, there is an $n_{\delta} \in \bb{N}$ such that the following is true for all $n>n_{\delta}$ and $i$ with $\delta \dfrac{\de_{n}}{(2n-2)} \leq i \leq (2n-3)!!$.
  
If $A \subseteq V(\Ma_{n})$ with $|A| = (2n-3)!!$ and $|A \setminus \calS_x | \geq i$ for every star $\calS_x$, then 
\[
\ed(A) \geq \kappa i^{2}.
\]
\end{theorem}

\begin{proof}
By Lemma~\ref{lem:specpmgraph}, $\tau= \dfrac{-\de_{n}}{2n-2}$ is the smallest eigenvalue of $\Ma_n$,
and if $\mu$ is the second smallest eigenvalues of $\Ma_{n}$, then $|\mu| = O((2n-5)!!)$.
So for $n$ sufficiently large, we have
\[
\tau -  \mu \geq \frac{\tau}{2}.
\] 
Applying Proposition~\ref{prop:edge-count-norm} in this case shows that
\begin{equation}\label{eq:edgespp}
\ed(A) \geq \dfrac{(2n-1)!! \de_{n}}{8n-8} \ \|\ind_{A}-\mrm{Proj}_{U}(\ind_{A})\|^2.
\end{equation}
As in the proof of Theorem~\ref{thm:Sn-edge-lb}, we apply Theorem~\ref{thm:EFFPM} in the contrapositive to prove that $\|\ind_{A}-\mrm{Proj}_{U}(\ind_{A})\|$ is large, which then implies $\ed(A)$ is large. Again we also start with the fact that $|A|=(2n-3)!!=|\calS_x|$ implies that $ | A \Delta S_x |  \geq 2i$ for every star $\calS_x$.

Using $C_{0}$ and $\varepsilon_{0}$ as in Theorem~\ref{thm:EFF-stability}, let $C_1>0$ be large enough so that
\begin{equation}\label{eq:C1pm}
 2C_{0} \leq C_{1}, \qquad  \frac{4}{C_1^2} \leq \varepsilon_0.    
\end{equation}

Again, we set $M=|\tau| = \frac{\de_n}{2n-2} $. Using \eqref{eq:numderpm}, we have $\lim\limits_{n \to \infty} \dfrac{M}{(2n-3)!!} =\dfrac{1}{\sqrt{e}}$. Therefore, when $n$ is sufficiently large, we have $\dfrac{M}{(2n-3)!!} \geq \dfrac{1}{2\sqrt{e}}$. We can therefore pick $n_{\delta}$ such that for all $n \geq n_\delta$
\begin{equation}\label{eq:ndpm}
 \dfrac{1}{2\sqrt{e}} \leq \dfrac{M}{(2n-3)!!} , \qquad
\dfrac{1}{2n-1} \leq    \dfrac{\delta}{\sqrt{e}C_{1}} .
\end{equation}

Set $\varepsilon = \left(\frac{2i}{C_1(2n-3)!!}\right)^2$. Since $i \leq (2n-3)!!$, by \eqref{eq:C1pm}, we have 
\[
\varepsilon=  \left(\frac{2i}{C_1(2n-3)!!}\right)^2 \leq \frac{4}{C_{1}^{2}} \leq \varepsilon_{0}.
\]
Using the two inequalities in~\eqref{eq:ndpm} and the fact that $i \geq \delta M$,
it follows that for all $n\geq n_{\delta}$, we have 
\begin{equation}\label{ineq:epsilonbound}
\dfrac{1}{2n-1} \leq   \dfrac{\delta}{\sqrt{e}C_{1}}
                         \leq  \dfrac{2 M \delta}{C_{1}(2n-3)!!}
                          \leq  \dfrac{2 i}{C_{1}(2n-3)!!}
                              = \varepsilon^{1/2}.  
\end{equation}
Using the above inequality and the first statement in~\eqref{eq:C1pm}, for any star $\calS_x$ we have 
\[
    |A \Delta \calS_x | \geq 2i  = C_{1}(2n-3)!! \varepsilon^{1/2} 
    \geq 2C_{0}(2n-3)!!\varepsilon^{1/2} 
     \geq C_{0}(2n-3)!! \left( \sqrt{\varepsilon} + \dfrac{1}{2n-2} \right).
\]
The final inequality follows from~\eqref{ineq:epsilonbound}.

The contrapositive of Theorem~\ref{thm:EFFPM} in the case $c=1$ implies that 
\[
||\ind_{A} - \Proj_U(\ind_{A})||^2 \geq \frac{\varepsilon}{2n-1}.
\] 
By~\eqref{eq:edgespp} 
\begin{align*}
\ed(A) &\geq \dfrac{(2n-1)!! \de_{n}}{8n-8} \ \frac{\varepsilon}{2n-1} \\
&= \dfrac{ (2n-3)!! \ \de_{n} \ 4i^2}{ 4 (2n-2) \ C_1^2  \ ( (2n-3)!!)^2  \ }   \\
&= \dfrac{ \de_{n}  i^2}{ (2n-2) \ C_1^2 \ (2n-3)!! }  \\
&= \dfrac{  M   i^2}{C_1^2  (2n-3)!!  }  \\
&\geq  \frac{1}{\sqrt{\varepsilon} C_1^2}  i^2 
\end{align*}
(where the last equation follows from~\eqref{eq:ndpm}). Thus the result holds with $\kappa =  1/ \sqrt{\varepsilon} C_1^2 $.
\end{proof}

\section{Proofs of Theorems~\ref{thm:randomekrsn} and \ref{thm:randomekrpm}.}

The EKR theorem for permutations or perfect matchings is equivalent to the result that every maximum independent set in $\Gamma_{n}$ or $\Ma_{n}$, respectively, is a star. In this section, we prove that the stars are still the only maximum independent sets in many sparse random spanning subgraphs of $\Gamma_{n}$ or $\Ma_{n}$. Before proceeding to the proof, we discuss some attributes shared by both families of graphs.

Throughout this section, $G_{n}$ will be one of $\{\Gamma_{n},\ \Ma_{n}\}$. We will use the following notation:
\begin{enumerate}
\item $V=V(n)$ is the number of vertices in $G_{n}$;
\item $d=d(n)$ is the valency of $G_{n}$;
\item $N=N(n) =\alpha(G_{n})$;
\item $M=M(n)= |\tau|$ is the absolute value of the smallest eigenvalue of $G_{n}$; and
\item $K=K(n)$ is the number of maximum independent sets in $G_{n}$. 
\end{enumerate}

In Table~\ref{table:params} we summarize the values of these parameters in the two situations. In the first column we list the parameter; in the second column we give its value in the context of $\Gamma_n$; in the third column we give its value in the context of $\Ma_n$, and in the final column we compare the value of the parameter to $N$.
By Theorems~\ref{thm:ekrsn} and \ref{thm:ekrpm}, $K$ is equal to the number of stars. 
We use Lemma~\ref{lem:specdergraph} and Lemma~\ref{lem:specpmgraph} for the final column. The approximations in rows 3 and 4 follow from~\eqref{eq:numderangements} and~\eqref{eq:numderpm}.

\begin{table}
\begin{center}
\begin{tabular}{|c|c|c|c| } \hline
Parameter & Permutations & Perfect matchings & Order\\ \hline
 $N$ & $(n-1)!$ & $(2n-3)!!$ &  \rownumber  \\
 $V$ & $n!$  & $(2n-1)!!$ & $\Theta(nN)$   \\    
 $d$ & $d_n \sim \frac{n!}{e}$ & $\frac{\de_n}{(2n-2)} \sim \frac{(2n-1)!!}{\sqrt{e}}$ & $\Theta(nN)$   \\
 $M$ & $\frac{d_n}{n-1}$ & $\frac{\de_n}{2n-2}$ & $\Theta(N)$ \\
 $K$ & $n^2$ & $\binom{n}{2}$ & $\Theta(n^2)$   \\ \hline   
\end{tabular}
\end{center}
\caption{Parameters for Derangement graphs\label{tbl:parameters}}\label{table:params}
\end{table}

Recall, as observed above, that the Delsarte--Hoffman ratio bound is tight for both these graphs.

If $H$ is a spanning subgraph of $G_{n}$, with $\alpha(G)<\alpha(H)$, then one of the following must hold: (a) there exists a star $\calS_x$ and a vertex $v$ outside $\calS_x$, such that $\calS_x \cup \{ v \}$ is independent in $H$; or 
(b) there is a set $\calA$ which is independent in $H$ and larger than $N$, but does not contain any stars.
To this end, we now define the following types of sets in $G_{n}$.

\begin{definition}\label{def:spsets}
Let $\calA$ be a set of vertices in $G_{n}$, then 
\begin{enumerate}
    \item $\calA$ is called a \emph{superstar} if there is a star $\calS_x$ and a vertex $v \notin \calS_x$ such that $\calA=\calS_x \cup \{v\}$;
    \item and $\calA$ is called a \emph{faux star} if $|\calA| > N$ and $\calS_x \setminus \calA \neq \emptyset$ for every star $\calS_x$ in $G_{n}$.
\end{enumerate}    
\end{definition}
If $H$ is a spanning subgraph of $G_{n}$, with $\alpha(G)<\alpha(H)$, then $H$ contains either an independent superstar or a faux star that is an independent set. 

Given a set $\calA$ of vertices, define the size of its largest intersection with a star to be 
\[
\partial \calA :=\mrm{max} \left\{ | \calA \cap \calS_x |\ :\  \calS_x  \text{ is a star} \right\}.
\]
We now introduce random variables that count the number of faux stars and independent superstars. For each $i$, define 
\[
X_{i} :=  |\left \{ \calA \ :\ \text{$\calA$ is an independent faux star with $\partial \calA =N-i$} \right\}| ,
\] 
and 
\[
Y= | \{ \calA \ : \ \calA~\text{is an independent superstar} \} |.
\] 
From our discussion above, we have
\begin{equation}\label{eq:unionbd}
\bb{P}\left[ \alpha(G_{n,p}) >N \right]
= 
\sum\limits_{i=1}^{N-1}\bb{P}\left[ X_{i}>0 \right] 
+  \bb{P}\left[ Y>0 \right].
\end{equation}

We now recall the isomperimetry results we derived from the previous section. 
Given a vertex $v$ and a star $\calS_x$, define $\calS_x[v]:=\calS_x \cup \{v\}$. Every superstar is of the form $\calS_x[v]$. Since the Delsarte--Hoffman ratio bound is tight for $G_{n}$, we have $\ed(\calS_x[v] )=M$. 
Tightness of this bound gives us yet another edge isoperimetry bound: given a set $T$ with $\partial T= N-i$, by Corollary~\ref{cor:edge-count-ratio}, we have $\ed(T)\geq i(M-i)$. 
Theorems~\ref{thm:Sn-edge-lb} and \ref{thm:pmedlb} give a lower bound on $\ed(A)$ in the case $i \geq M$.

We will first investigate the threshold for the appearance of a superstar in $G_{n,p}$. 
The probability of set of vertices $A$ being independent in $G_{n,p}$ is $(1-p)^{\ed(A)}$, so the probability that a superstar is an independent set in $G_{n,p}$ is $(1-p)^{M}$. There are exactly $K$ stars, the size of every star is $N$, and every superstar spans exactly $M$ edges, so we have that
\begin{equation}\label{eq:EY}
\bb{E}[Y] = K (V-N) (1-p)^{M}. 
\end{equation}

We now define our threshold probability 
\begin{equation*}\label{eq:thprob}
 p_{c}:=p_{c}(n)= \dfrac{\ln(K(V-N))}{M}.   
\end{equation*}
We note that $p_{c}$ matches the threshold probabilities defined in Theorems~\ref{thm:randomekrsn} and \ref{thm:randomekrpm}. Given $\varepsilon>0$, we now show that when $p\geq (1+\varepsilon)p_{c}$, then $\bb{P}\left[ Y>0 \right] =o(1)$. We will make use of a well-know inequality that we state below.

\begin{lemma}\label{lem:expbound}
For every $x \in \mathbb{R}$ with $|x| \leq 1/2$
\[
e^{x-x^2} \leq 1+x  \leq e^{x} .
\]
\end{lemma}

\begin{lemma}\label{lem:Yub}
Let $\varepsilon>0$ and consider $p \geq (1+\varepsilon)p_{c}$. Then 
\[\lim\limits_{n \to  \infty} \bb{P}[ Y>0 ]  =0.\]
\end{lemma}
\begin{proof}
    From Chebyshev's inequality, we have 
    $\bb{P}[ Y >0 ] \leq \bb{E}[Y]$. Using $e^{-x} \geq 1-x$ and \eqref{eq:EY}, we have   
    \begin{align*}
        \bb{P}[ Y>0 ] &\leq \bb{E}[Y] \\
        &\leq K (V-N) (1-p)^{M} \\
        & \leq K (V-N) e^{-pM} \\
        & \leq K (V-N) e^{-(1+\varepsilon)p_{c}M} \\
        & \leq K(V-N) e^{-(1+\varepsilon)\ln(K(V-N))} \\
        & \leq (K(V-N))^{-\varepsilon}.
    \end{align*}
 The result now follows from Table~\ref{tbl:parameters}. 
\end{proof}

We now show that when $p$ is smaller than the threshold $p_{c}$, independent superstars appear in $G_{n,p}$, with high probability. Thus, when $p$ is smaller than $p_{c}$, we have $\alpha(G_{n,p})>N$, with high probability. To prove this, we first establish the following bound.

\begin{lemma}\label{lem:KVNbound}
    Given any $\delta>0$, we have $K=o((V-N)^{\delta})$.
\end{lemma}
\begin{proof}
   To show this, we use two relaxations of Stirling's approximation
   \[
    \left(\dfrac{r}{e} \right)^{r} \leq r!,
\qquad
 \frac{1}{\sqrt{2}} \sqrt{2\pi r} \left(  \dfrac{r}{e} \right)^{r} 
 \leq r! \leq 
   \sqrt{2}   \sqrt{2\pi r} \left( \dfrac{r}{e} \right)^{r} .
   \]

When $G_{n}=\Gamma_{n}$, we have 
\[
    \dfrac{K}{(V-N)^{\delta}}  = \dfrac{n^{2}}{(n(n-1)!)^{\delta}} 
     \leq \dfrac{n^{2}}{\left((n-1)\dfrac{(n-1)^{n-1}}{e^{n-1}}\right)^{\delta}} 
     \leq \dfrac{n^{2}}{\left(\dfrac{(n-1)}{e} \right)^{\delta(n-1)}}
\]
When $n$ is sufficiently large, we have $\delta(n-1)>3$, and thus when $n$ is sufficiently large, we have 
\[
\dfrac{K}{(V-N)^{\delta}} \leq \dfrac{n^{2}}{\left(\dfrac{(n-1)}{e} \right)^{\delta(n-1)}} 
                                       \leq \dfrac{n^{2}}{\dfrac{(n-1)^{3}}{e^{3}}}
                                       =o(1).
\]

Now we consider the case $G_{n}=\Ma_{n}$. In this case, $K=\binom{2n}{2}=(n)(2n-1)$ and $V-N=(2n-2)(2n-3)!!$. Using Stirling's approximation, we have
\begin{align}\label{eq:dfacster}
    (2n-3)!! & = \dfrac{(2n-2)!}{ ( 2^{n-1}) \ (n-1)!} \notag \\
    & \geq \dfrac{\sqrt{2\pi (2n-2)} \left(\frac{(2n-2)}{e} \right)^{2n-2}}{ ( 2^{n-1}) \ 2 \sqrt{2\pi (n-1)} \left(\frac{(n-1)}{e} \right)^{n-1}} \notag \\
    & \geq \left(\dfrac{2(n-1)}{e} \right)^{n-1}.
\end{align}
In this case, we have 
\[
\dfrac{K}{(V-N)^{\delta}} = \dfrac{n(2n-1)}{((2n-2)(2n-3)!!)^{\delta}} 
                                     \leq \dfrac{2n^{2}}{(2n-3)!!)^{\delta}} 
                                     \leq \dfrac{2n^{2}}{\left(\dfrac{(2n-2)}{e}\right)^{\delta(n-1)}}.
\]
When $n$ is sufficiently large, we have $\delta(n-1)>3$, and thus when $n$ is sufficiently large, we have 
\[
\dfrac{K}{(V-N)^{\delta}} \leq \dfrac{2n^{2}}{\left(\dfrac{(2n-2)}{e}\right)^{\delta(n-1)}} 
                                       \leq \dfrac{2n^{2}}{\left(\dfrac{(2n-2)}{e}\right)^{3}} 
                                       =o(1).
\]
\end{proof}
We also need the following bound.
\begin{lemma}\label{lem:pcbound}
    $p_{c}=o(1).$
\end{lemma}
\begin{proof}
    Using Table~\ref{tbl:parameters}, provided $n$ is sufficiently large, we have $M\geq aN$, for some absolute constant $a$, and $K(V-N)\leq bn^{3}N$, for some absolute constant $b$. Thus, when $n$ is sufficiently large 
    \[
    p_{c}= \dfrac{\ln(K(V-N))}{M} \
           \leq \dfrac{3\ln(n)+ \ln(b) + \ln(N)}{a N}.
    \]
    In the case $G_{n}=\Gamma_{n}$ we have $N=(n-1)!$ and for $G_{n}=\Ma_{n}$, $N=(2n-3)!!$. 
In either case both 
\[
\lim_{n\to \infty }\dfrac{\ln(n)}{N}=0, \quad \lim_{n\to \infty } \dfrac{\ln(N)}{aN} = 0
\] 
hold.
As $b$ is a constant, these imply 
$p_{c}=o(1)$ in both cases. 
\end{proof}

We are now ready to show that when $p$ is smaller that $p_{c}$, independent superstars appear with high probability. 
\begin{lemma}\label{lem:pclb}
    Let $\varepsilon>0$ and consider $p<(1-\varepsilon)p_{c}$. Then we have 
    \[\lim\limits_{n\to \infty} \bb{P}\left[ \alpha(G_{n,p}) =N \right] =0,\] and with high probability, $G_{n,p}$ contains an independent superstar.
\end{lemma}
\begin{proof}
As $p_{c}=o(1)$, for sufficiently large $n$, we have the following two bounds on $p_c$:
\begin{equation}\label{eq:pcassump}
p_{c}\leq \frac{1}{2}, \quad (1-\varepsilon)(1+p_{c}) \leq (1-\varepsilon/2).
\end{equation}
 We may assume the above conditions, without loss of generality.

 Given a vertex $v$, not contained in a star $\calS_x$, consider the superstar $\calS_x[v]$. From Theorem~\ref{thm:ratiobound} $\ed(\calS_x[v])=M$, so 
 \[
 \bb{P} \left[ \calS_x[v] \text{ is independent } \right]  =(1-p)^{M}.
 \]
As $p \leq (1-\varepsilon)p_{c} \leq p_{c} \leq \frac{1}{2}$, we have  $(1-p) \geq e^{-p-p^{2}}$, and thus,
\begin{align}\label{eq:prindsslb}
    \bb{P}\left[  \calS_x[v]  \text{ is independent } \right] & =(1-p)^{M} \notag\\
                                                          &\geq e^{-pM(1+p)} \notag \\ 
                                                          & \geq e^{-(1-\varepsilon)p_{c} M (1+p_{c})}  \notag\\
                                                          & \geq e^{-(1-\frac{\varepsilon}{2})p_{c}M}\quad \text{(using~\eqref{eq:pcassump})} \notag \\
                                                          & \geq (K(V-N))^{-(1-\varepsilon/2)}.
\end{align}

Fix a star $\calS_x[v]$. The events $\{ \calS_x[v] \text{ is independent} \}$ are mutually independent across all choices of $v \notin \calS$. Therefore, we have 
\begin{align*}
    \bb{P}\left[ \alpha(G_{n,p}) =N \right] 
    & \leq \bb{P}\left[ \calS_x[v] \text{ is not independent for all } v \notin \calS  \right] \\
    & \leq \prod\limits_{v \notin \calS} \left( 1 - \bb{P} \left[ \calS_x[v] \text{ is independent} \right]  \right) \\
    & \leq \left(1-(K(V-N))^{-(1-\varepsilon/2)}\right)^{V-N} \quad \text{(using \eqref{eq:prindsslb})} \\
    &\leq \mrm{exp}\left( - (K(V-N))^{-(1-\varepsilon/2)} (V-N) \right)\quad \text{(since $e^{-x}\geq 1-x$)} \\
    & \leq \mrm{exp} \left( -\dfrac{(V-N)^{\varepsilon/2}}{K^{(1-\varepsilon/2)}} \right) \\
    & \leq \mrm{exp} \left( -\dfrac{(V-N)^{\varepsilon/2}}{K} \right) \\
    &=o(1),
\end{align*}
 with the equality following from Lemma~\ref{lem:KVNbound}. This concludes the proof. 
\end{proof}
So far, we proved that when $p<(1-\varepsilon)p_{c}$, we have $\lim\limits_{n\to \infty} \bb{P}\left[ \alpha(G_{n,p}) =N \right] =0$. We also proved that in the case $p\geq (1+\varepsilon)p_{c}$, $G_{n,p}$ has no independent superstars with high probability. To finish proving that $\lim\limits_{n\to \infty} \bb{P}\left [ \alpha(G_{n,p}) =N \right] =1$ in the case $p\geq (1+\varepsilon)p_{c}$, we need to show that with high probability, $G_{n,p}$ does not contain any faux stars (see Definition~\ref{def:spsets}) either. Given an integer $i$, we recall that $X_{i}$ is the random variable that counts the number of independent faux stars $\calA$ with $\partial A =N-i$. For ease of notation, by a faux star of type $i$, we mean a faux star $\calA$ with $\partial \calA =N-i$. To finish our proof, we need to show that $\bb{P}[ X_{i}>0 ] =o(1)$, provided $p \geq (1+\varepsilon)p_{c}$.

At first, we consider faux stars which are ``far'' from being stars. To be precise, we consider faux stars of type $i \geq \dfrac{\varepsilon M}{2}$. 

\begin{lemma}\label{lem:pmxlarge}
For every $\varepsilon>0$, if $p \geq (1+\varepsilon)p_{c}(n)$, then 
\[
\lim\limits_{n \to \infty}  \bb{P} [ X_i > 0 ] =0,
\] 
for all $i \geq \dfrac{\varepsilon M}{2}$.
\end{lemma}

\begin{proof}
Using Stirling's approximation, we have
\[
\binom{s}{t} \leq \left(\dfrac{s e}{t} \right)^{i}.
\] 
Due to the monotonicity of independence, we may assume that $p=(1+\varepsilon)p_{c}$. Rough counting estimates imply that the number of faux stars of type $i\geq \frac{\varepsilon M }{2}$ is at most 
\begin{align*}
    K \binom{N}{i}\binom{V-N}{i} &\leq  K 2^{N} \left(\dfrac{(V-N)e}{i} \right)^{i} \\
    & \leq K 2^{N} \left(\dfrac{2(V-N)e}{\varepsilon M} \right)^{i}.
\end{align*}
Using the values in Table~\ref{tbl:parameters}, there is an absolute constant $b$ such that $\dfrac{2(V-N)e}{\varepsilon M} \leq b n$, and an absolute constant $a$ such that $K \leq an^{2}$. Therefore, the number of faux stars of type $i$ is at most $an^{2} 2^{N} \left( b n\right)^{i}$. 

Using Theorem~\ref{thm:Sn-edge-lb} and Theorem~\ref{thm:pmedlb}, any faux star of type $i \geq \varepsilon M/2$ spans at least $\kappa i^{2}$ edges (where $\kappa$ is an absolute constant), provided $n$ is sufficiently large. 
We can now conclude that 
\begin{align*}
    \bb{E}[X_{i}] & \leq an^{2} 2^{N} \left(b_n \right)^{i} (1-(1+\varepsilon)p_{c})^{\kappa i ^2} \\
    & \leq an^{2} 2^{N} \left(bn \right)^{i} (1-(1+\varepsilon)p_{c})^{\kappa \varepsilon M i/2} \quad \text{(since $i\geq \varepsilon M$)} \\
    & \leq an^{2} 2^{N} \left(bn  \right)^{i}  \exp \left(-(1+\varepsilon) p_{c} \kappa \varepsilon M /2 \right)^{i} \\
    & \leq an^{2} 2^{N} \left(bn \right)^{i} \exp \left(-(1+\varepsilon)\varepsilon \kappa \ln(K(V-N)) /2 \right)^{i} \\
    & \leq an^{2} 2^{N} \left(bn  (K(V-N))^{-(1+\varepsilon) \varepsilon\kappa/2}\right) ^{i} \\
    & \leq an^{2} 2^{N} \left( \frac{ b n }{ (K(V-N))^{\varepsilon^{2}\kappa/2} } (K(V-N))^{-\varepsilon\kappa/2}\right) ^{i} \\
    &\leq an^{2} 2^{N} \left( (K(V-N))^{-\varepsilon\kappa/2}\right)^{i},
\end{align*}
with the last equation following from the fact $n=o(K(V-N)^{\delta})$, for all $\delta >0$.

Thus, 
\begin{align*}
    \sum\limits_{i= \varepsilon M /2} ^{N} \bb{E}[X_{i}] & \leq \sum\limits_{i= \varepsilon M /2} ^{N} an^{2} 2^{N} \left( (K(V-N))^{-\varepsilon\kappa/2}\right) ^{i} \\ 
    & \leq an^{2} 2^{N} \left( (K(V-N))\right)^{-\varepsilon^{2}\kappa M/2}.
\end{align*}
Since $N=\Theta(M)$, $K=\Theta(n^{2})$, and since $\lim\limits_{n\to \infty}((V-N))^{\delta}=\infty$, for all $\delta>0$, we can conclude that $\sum\limits_{i= \varepsilon M /2} ^{N} \bb{E}[X_{i}]=o(1)$. The result now follows by an application of Chebyshev's inequality.
\end{proof}

Now we consider the probability of having faux stars of type $i < \varepsilon M/2$. There are at most $K\binom{N}{i}\binom{V-N}{i}$  faux stars of type $i$. 
Given a faux star $A$ of type $i<\varepsilon M/2$, by Corollary~\ref{cor:edge-count-ratio}, we have $\ed(A)\geq i(M-i)$. By Chebyshev's inequality, we have 
\[\bb{P}[ X_{i}>0 ] \leq \bb{E}[X_{i}] \leq K\binom{N}{i}\binom{V-N}{i} (1-p)^{i(M-i)}.\] However, the above bound is too large to show that 
$\lim\limits_{n \to \infty}   \bb{P}[ X_{i} ] =o(1) $, and for this reason, we adopt a different strategy. 

 Given $j\geq i$, we define the random variable
\[
X_{i,j} = | \left\{ \calA \ :\ \text{$|\calA|=N+j-i$, $\partial \calA  =N-i$,  $\calA$ is a maximal independent set} \right\} |.
\] 
We have
\[
\bb{P}[ X_{i} >0 ] = \sum\limits_{j\geq i}\bb{P}[ X_{i,j} > 0]  
                       \leq \sum\limits_{j\geq i}\bb{E}[X_{i,j}].
\] 
To find a bound on $\bb{E}[X_{i,j}]$, we use two approximations. The first is a lower bound on the number of edges in $\calA$. This bound follows from Corollary~\ref{cor:edge-count-ratio}, which implies  that $\ed(\calA) \geq j(M-i)$. The second is a lower bound on the probability that $\calA$ is a \emph{maximal} independent set in $G_{n,p}$. We find this bound by counting the number of edges from any vertex in $\calS_x \backslash \calA$ to $\calA$, where $\calS_x$ is a star. 
If $\calA$ is a maximal, as well as independent, set in $G_{n,p}$, then at least one of these edges must remain in $G_{n,p}$. 
Since $\partial \calA =N-i$, there is a star $\calS_x$ and sets $I \subset \calS_x$ and $J$ with $J \cap \calS_x=\emptyset$ with $| I |=i$ and $|J|=j$, such that $\calA=\left( \calS_x \setminus I \right) \cup J$. Consider a $v \in \calS_x$ then
\begin{align*}
\bb{P} \left [ v \text{ has at least one edge incident to $J$}  \right] 
	\leq 1 - (1-p)^j 
	\leq jp.
\end{align*}
Using the above two bounds
\[
\bb{P} \left[ \text{$\calA$ being a maximal independent set in $\Gamma_{n,p}$} \right] \leq (1-p)^{j(M-i)} (jp)^{i}.
\] 
By counting the number of stars and the number of sets $\calA$ of size $N+j-i$ with $\partial(\calA) = N-i$ we can now conclude that 
\begin{equation}\label{Eij}
\bb{E} \left[ X_{i,j} \right]  \leq  K \binom{N}{i}\binom{V}{j} (1-p)^{j(M-i)} (jp)^{i}.
\end{equation}
Define $a_{i,j} = K \binom{N}{i}\binom{V}{j} (1-p)^{j(M-i)} (jp)^{i}$, then 
\[
\frac{\alpha_{i,j+1} } { \alpha_{i,j} } \leq V(1-p)^{M-i}(1+1/j)^{i}.
\]

As the property of having an independent set of a given size is monotone decreasing, it suffices to assume that 
$p=(1+\varepsilon)p_{c}$, with $0<\varepsilon <1$. 
Using $1-x \leq e^{-x}$, along with the fact that $i \leq \varepsilon M/2$ and $(1+1/j)^{i}$ is bounded by $e$, we have 
\begin{align*}
V(1-p)^{M-i}(1+1/j)^{i} & \leq V e^{-(M-\varepsilon M /2)p}e\\
& \leq V e^{(\varepsilon^{2}/2 -\varepsilon/2 - 1)p_{c}M}e\\
& \leq V \left(K(V-N) \right)^{(\varepsilon^{2}/2 -\varepsilon/2 - 1)}e.
\end{align*}
Set
\begin{align*}
r_{n,\varepsilon} =V  \left( K(V-N) \right)^{(\varepsilon^{2}/2 -\varepsilon/2 - 1)}e, 
\end{align*}
so $\alpha_{i,j}\leq  \alpha_{i,j - 1}r_{n,\varepsilon}$, and doing this repeatedly, $\alpha_{i,j}\leq  \alpha_{i,i}r_{n,\varepsilon}^{j-i}$.

From Table~\eqref{tbl:parameters}, $K(V-N)=\Theta(n^{3}N)$ and $V= \Theta(nN)$. Using these bounds along with the fact that $\varepsilon^{2}/2 -\varepsilon <0$ (recall that $\varepsilon \in (0,1)$ by assumption), we conclude that $r_{n,\varepsilon}=o(1)$. Thus, provided $n$ is large enough, we have $r_{n,\ \varepsilon}<1/2$, which implies

\begin{equation}\label{eq:pxi}
\bb{P}[ X_i>0 ]
\leq \sum\limits_{j\geq i}\bb{E}[X_{i, j}] 
\leq \sum\limits_{j\geq i} \alpha_{i,j}
= \frac{1}{1-r_{n,\varepsilon}}\alpha_{i,i} 
\leq 2 \alpha_{i, i}.
\end{equation}

We now turn our focus to $\alpha_{i,i}$. Using $\binom{a}{b} \leq \left(\frac{ae}{b} \right)^{b}$ and $e^{-x}\geq 1-x$, we have 
\begin{align*}
\alpha_{i,i} =K \binom{N}{i} \binom{V}{i} (1-p)^{i(M-i)} (ip)^{i}
\leq  K\left(\frac{e^{2}VN\exp(-p(M-i))p}{i} \right)^{i}.
\end{align*}

As $M-i \geq (1-\varepsilon/2)M$ and $p=(1+\varepsilon)p_{c}=(1+\varepsilon)\ln\left(K(V-N)\right)/M$, we have
\begin{align*}
\alpha_{i,i} 
	& \leq K\left(\frac{e^{2}VN\exp\left[-(1-\varepsilon^{2}/2+\varepsilon/2)\ln\left(K(V-N)\right)\right](1+\varepsilon)\ln\left(K(V-N)\right)}{i \ M} \right)^{i}\\
	& \leq K\left(\frac{e^{2}VN \left( (K(V-N)^{(-1+\varepsilon^{2}/2-\varepsilon/2)} \right) 2 \ln\left(K(V-N)\right)}{ M} \right)^{i},
	              \quad \text{using $1+\varepsilon < 2$ and $i\geq 1$} \\
	&\leq K\left(\frac{2e^{2}VN \left( K(V-N) \right)^{\varepsilon^{2}/2-\varepsilon/2} 
	                                \ln\left(K(V-N)\right)}{K(V-N) \ M }\right)^{i}.
\end{align*}
 
 Again, since $V=\Theta(nN)$, $K(V-N)= \theta(n^{3}N)$, $K=\theta(n^{2})$ and $M=\Theta(N)$, there are absolute constants  $a$, $b$ such that 
\[
\alpha_{ii} \leq  a n^{2} \left(\dfrac{b \ln(n^{3}N)}{n^{2} (n^{3}N)^{\varepsilon/2 -\varepsilon^{2}/2}} \right)^{i}.
\]

As $N$ is either $(n-1)!$ (when $G_{n}=\Gamma_{n}$) or $N=(2n-3)!!$ (when $G_{n}=\Ma_{n}$), we have $\ln(n^{3}N)= o( (n^{3}N)^{\delta})$, for all $\delta>0$.
In particular, we can set $\delta = \varepsilon/2 -\varepsilon^{2}/2-\varepsilon/4$ to get 
\[
\frac{ \ln(n^{3} N)(n^{3}N)^{\varepsilon/4} }{ (n^{3}N)^{\varepsilon/2 -\varepsilon^{2}/2} } = \frac{ \ln(n^{3} N) }{ (n^{3}N)^{\varepsilon/2 -\varepsilon^{2}/2-\varepsilon/4} } \leq 1
\]
This, when rearranged, implies that as long as $n$ is sufficiently large, we have
\[
\alpha_{i,i} \leq a n^{2} \left(\dfrac{b \ln(n^{3}N)}{n^{2} (n^{3}N)^{\varepsilon/2 -\varepsilon^{2}/2}} \right)^{i} 
                  \leq a n^{2} \left(\dfrac{b}{n^{2} (n^{3}N)^{\varepsilon/4}} \right)^{i}
                  = a n^{2} \left(\dfrac{b}{n^{2+\varepsilon/2} (nN)^{\varepsilon/4}} \right)^{i}.
\]
Thus, for $n$ sufficiently large, $\dfrac{b}{n^{2+\varepsilon/2} (nN)^{\varepsilon/4}} <1/2$. Now using \eqref{eq:pxi}, we have 
\begin{align*}
\sum\limits_{i=1}^{\varepsilon M/2} \bb{P}[X_i > 0]  
     &\leq 2\sum\limits_{i=1}^{\varepsilon M /2} \alpha_{ii}\\
     &\leq 2\sum\limits_{i=1}^{\varepsilon M /2} a n^{2} \left(\dfrac{b}{n^{2+\varepsilon/2} (nN)^{\varepsilon/4}} \right)^{i} \\
     & \leq 4 a n^{2} \left(\dfrac{b}{n^{2+\varepsilon/2} (nN)^{\varepsilon/4}} \right) =o(1),
\end{align*}
(the third inequality is from the fact that $\sum_{i=1}^\infty x^i \leq 2x$ if $0< x \leq 1/2$).

We have now shown that 

\begin{lemma}\label{lem:pmxsmall}
For every $\varepsilon>0$, if $p \geq (1+\varepsilon)p_{c}(n)$, then 
\[
\lim\limits_{n \to \infty} \bb{P} [ X_i > 0 ] =0,
\] 
for all $i < \dfrac{\varepsilon M}{2}$.
\end{lemma}

Using \eqref{eq:unionbd} along with Lemmas~\ref{lem:Yub}, \ref{lem:pmxlarge}, and~\ref{lem:pmxsmall}, we can conclude that 
\begin{lemma}\label{lem:pcub}
    Let $\varepsilon>0$ and $p\geq (1+\varepsilon)p_{c}$. Then we have 
    \[\lim\limits_{n\to \infty} \bb{P}\left[  \alpha(G_{n,p}) =N \right] =1,\] and with high probability, every maximum independent set in $G_{n,p}$ is a star.
\end{lemma}

Theorems~\ref{thm:randomekrsn} and \ref{thm:randomekrpm} follow from Lemmas~\ref{lem:pclb} and \ref{lem:pcub}.
\section{Proof of Theorem~\ref{thm:EFFPM}.}
\label{sec:EFFPM}

In this section we prove Theorem~\ref{thm:EFFPM}. Recall that $U$ is the space formed by the span of the characteristic vectors of the stars in $\Ma_n$. This result shows that if a relatively large set $A$ of vertices in the graph $\Ma_n$ has the property that the characteristic vector of $A$ ($\ind_A$), and the projection of $\ind_A$
 to $U$
are close (in norm), then the set $A$ is close (in terms of the symmetric difference) to a star.

Before starting the proof we give equations for the size of intersection of two stars in $\Ma_n$.
For pairwise distinct edges $e_{1},e_{2},e_{3} \in E(\mcal{K}_{2n})$ (so that $|e_i \cap e_j| \in \{0,1\}$ for $1 \le i,j \le 3$ with $i \neq j$), we have 
\begin{equation}\label{eq:pmstarint}
\begin{split}
|\calS_{e_{1}} \cap \calS_{e_{2}}| &= (2n-5)!!(1-|e_{1}\cap e_{2}|), \\ 
  |\calS_{e_{1}} \cap \calS_{e_{2}}\cap \calS_{e_{3}}| &=(2n-7)!!(1-|e_{1}\cap e_{2}|) (1-|e_{1}\cap e_{3}|)(1-|e_{2}\cap e_{3}|),
\end{split}
\end{equation}
since a matching cannot contain two distinct edges that have nonempty intersection.

Throughout this section, we let $A$ be a set of perfect matchings of size $c(2n-3)!!$ and $f:=\ind_{A}$ denote the indicator function of $A$. When using a variable to denote a single arbitrary edge we use $\mathsf{e}$ to avoid possible confusion with Euler's constant which we have also been using in this paper.
Given $\mathsf{e} \in E(\mcal{K}_{2n})$, define 
\[
a_{\mathsf{e}}:=\dfrac{|A\cap \calS_{\mathsf{e}}|}{(2n-3)!!}, \quad b_{\mathsf{e}}:=a_{\mathsf{e}}-\frac{c}{2n-1}.
\]
Given a vertex $x$ of $\mcal{K}_{2n}$, each perfect matching in $A$ has $x$ matched with exactly one other vertex in $\mcal{K}_{2n}$,
so the set $A$ can be partitioned into the sets $A \cap \calS_{\mathsf{e}} $ where $x \in \mathsf{e}$. Thus
\begin{equation}\label{eq:asum}
\sum\limits_{\{\mathsf{e}\ :\ x \in \mathsf{e}\}} a_{\mathsf{e}} =c, \quad  \sum\limits_{\{\mathsf{e}\ :\ x \in \mathsf{e}\}} b_{\mathsf{e}} =0.
\end{equation} 

If $A$ is a disjoint union of stars, we observe that the value of each $b_{\mathsf{e}}$ will either be close to $1$ or close to $0$. We will prove that this is also true if $A$ is a set as described in Theorem~\ref{thm:EFFPM}. We can do so by showing that $\sum\limits_{\mathsf{e} \in E(\mcal{K}_{2n})} b^{2}_{\mathsf{e}}-b^3_{\mathsf{e}}$ is ``small''. To do so, we first find a relationship between moments of $f_{1}:=\Proj_{U}( \ind_{A} )$ and $\sum\limits_{\mathsf{e} \in E(\mcal{K}_{2n})} b^{2}_{\mathsf{e}}-b^3_{\mathsf{e}}$.

We define 
\[
g :=\sum\limits_{\mathsf{e}\in E(\mcal{K}_{2n})} a_{\mathsf{e}}\ind_{\calS_{\mathsf{e}}},
\qquad
h :=\sum\limits_{\mathsf{e}\in E(\mcal{K}_{2n})} b_{\mathsf{e}}\ind_{\calS_{\mathsf{e}}}.
\]
Our first result towards finding the relationship we seek is a bound on the expectation of $h^2$.

\begin{lemma}\label{lem:b2}
For any $n$
\begin{equation*}
\bb{E}[h^2] = \frac{(2n-2)}{(2n-1)(2n-3)} \sum\limits_{\mathsf{e} \in E(\mcal{K}_{2n})} b^{2}_{\mathsf{e}}.
\end{equation*}
\end{lemma}
\begin{proof}
Using \eqref{eq:pmstarint}, we see that 
\begin{align*}
h^2 &= \sum\limits_{\mathsf{e} \in E(\mcal{K}_{2n})} b^{2}_{\mathsf{e}} \ind_{\calS_{\mathsf{e}}} + \sum\limits_{\substack{\mathsf{e},\mathsf{f} \in E(\mcal{K}_{2n})\\ \mathsf{e}\cap \mathsf{f}=\emptyset}} b_{\mathsf{e}}b_{\mathsf{f}} \ind_{\calS_{\mathsf{e}}\cap \calS_{\mathsf{f}}}. 
\end{align*}
Computing the expectations yields
\begin{align}\label{eq:b21}
\bb{E}[h^{2}] = \frac{1}{2n-1} \sum\limits_{\mathsf{e} \in E(\mcal{K}_{2n})} b^{2}_{\mathsf{e}} + \frac{1}{(2n-1)(2n-3)} \sum\limits_{\substack{\mathsf{e},\mathsf{f} \in E(\mcal{K}_{2n})\\ \mathsf{e}\cap \mathsf{f}=\emptyset}} b_{\mathsf{e}}b_{\mathsf{f}}.
\end{align}
In the next equation, we denote the edges of $\mcal{K}_{2n}$ by their endpoints, 
and use the fact that the sum of all $b_{\mathsf{e}}$ over all edges $\mathsf{e}$ that contain $x$ is equal to 0. Fix an edge $\{x,y\}$, then 
\begin{align*}
2 \sum_{\stackrel{ \{u,v\} }{ x,y \not \in \{u,v\}} } b_{\{u,v\}} &=  - \sum_{v \neq x,y} b_{\{ x,v \} }  - \sum_{u \neq x,y} b_{ \{y,u\} }  \\
&=  - \sum_{v \neq x} b_{\{ x,v \} }  - \sum_{u \neq y} b_{ \{y,u\} }  +2 b_{\{ x,y\} }\\
&= 2b_{ \{ x,y \} }.
\end{align*}
Thus 
\begin{equation}\label{eq:bsumdisjoint}
 \sum_{\stackrel{  \mathsf{f}  }{ \mathsf{f} \cap \mathsf{e} = \emptyset } } b_{ \mathsf{f}}  = b_{\mathsf{e} },
\end{equation}
and we have
\begin{align}\label{eq:b22}
\sum_{\stackrel {\mathsf{e},\mathsf{f} \in E(\mathcal{K}_{2n})} {\mathsf{e} \cap \mathsf{f} = \emptyset} } b_{\mathsf{e}} b_{\mathsf{f}}
= \sum_{  \mathsf{e} \in E(\mathcal{K}_{2n}) } b_{\mathsf{e}}  
    \sum_{\stackrel{  \mathsf{f}  }{ \mathsf{f} \cap \mathsf{e} = \emptyset } } b_{ \mathsf{f} }   =   \sum_{ \mathsf{e} \in E(\mathcal{K}_{2n}) } b_{\mathsf{e}}^2. 
\end{align}
Putting this into Equation~\eqref{eq:b21} yields the equation in the statement of the lemma.
\end{proof}

Next we establish a bound on the expectation of $h^3$, namely Lemma~\ref{lem:thirdmomentbound}. We will arrive at this bound in two steps. We begin by bounding $\bb{E}[h^{3}]$ in terms of $b_{\mathsf{d}}b_{\mathsf{e}}b_{\mathsf{f}}$. After this (in Lemma~\ref{lem:bdbebfBound}) we will bound $b_{\mathsf{d}}b_{\mathsf{e}}b_{\mathsf{f}}$.

\begin{lemma}\label{lem:b33}
For all $n$,
\begin{equation*}
\begin{split}
\bb{E}[h^{3}] = & \frac{2(n-2)}{(2n-3)(2n-5)} \sum\limits_{\mathsf{e} \in E(\mcal{K}_{2n})} b^{3}_{\mathsf{e}}\\
& -\frac{1}{(2n-1)(2n-3)(2n-5)}
\left[\sum\limits_{\substack{\mathsf{d},\mathsf{e},\mathsf{f} \in E(\mcal{K}_{2n})\\ |\mathsf{d} \cap \mathsf{e}| =1,\ |\mathsf{d} \cap \mathsf{f}|=1, \ |\mathsf{e}\cap \mathsf{f}|=1 }}b_{\mathsf{d}}b_{\mathsf{e}}b_{\mathsf{f}} \right]. 
\end{split}
\end{equation*}
\end{lemma}
\begin{proof}
Again using \eqref{eq:pmstarint}, we see that 
\begin{align*}
h^{3} &= \sum\limits_{\mathsf{e} \in E(\mcal{K}_{2n})} b^{3}_{\mathsf{e}} \ind_{\calS_{\mathsf{e}}} + 3\sum\limits_{\substack{\mathsf{e},\mathsf{f} \in E(\mcal{K}_{2n})\\ \mathsf{e}\cap \mathsf{f}=\emptyset}} b^2_{\mathsf{e}}b_{\mathsf{f}} \ind_{\calS_{\mathsf{e}}\cap \calS_{\mathsf{f}}} + \sum\limits_{\substack{\mathsf{d},\mathsf{e},\mathsf{f} \in E(\mcal{K}_{2n})\\ \mathsf{d}\cap \mathsf{e}=\mathsf{e}\cap \mathsf{f}=\mathsf{d}\cap \mathsf{f}=\emptyset}} b_{\mathsf{d}}b_{\mathsf{e}}b_{\mathsf{f}} \ind_{\calS_{\mathsf{d}}\cap \calS_{\mathsf{e}}\cap \calS_{\mathsf{f}}},
\end{align*}
and 
\begin{align}\label{eq:b31}
\begin{split}
\bb{E}[h^{3}] =& \frac{1}{2n-1}\sum\limits_{\mathsf{e} \in E(\mcal{K}_{2n})} b^{3}_{\mathsf{e}}  + \frac{3}{(2n-1)(2n-3)}\sum\limits_{\substack{\mathsf{e},\mathsf{f} \in E(\mcal{K}_{2n})\\ \mathsf{e}\cap \mathsf{f}=\emptyset}} b^2_{\mathsf{e}}b_{\mathsf{f}}  \\ 
         &+ \frac{1}{(2n-1)(2n-3)(2n-5)}\sum\limits_{\substack{\mathsf{d},\mathsf{e},\mathsf{f}\in E(\mcal{K}_{2n})\\ \mathsf{d}\cap \mathsf{e}=\mathsf{e}\cap \mathsf{f}=\mathsf{d}\cap \mathsf{f}=\emptyset}} b_{\mathsf{d}}b_{\mathsf{e}}b_{\mathsf{f}}.
\end{split} 
\end{align}

Using \eqref{eq:bsumdisjoint}, 
\begin{align}\label{eq:b32}
 \sum\limits_{\substack{\mathsf{e},\mathsf{f} \in E(\mcal{K}_{2n})\\ \mathsf{e}\cap \mathsf{f}=\emptyset}} b^2_{\mathsf{e}}b_{\mathsf{f}}  
 = \sum\limits_{\mathsf{e} \in E(\mcal{K}_{2n})}b^{2}_{\mathsf{e}} \sum\limits_{\substack{\mathsf{f} \in E(\mcal{K}_{2n})\\ \mathsf{e}\cap \mathsf{f}=\emptyset}} b_{\mathsf{f}}  
 = \sum\limits_{\mathsf{e} \in E(\mcal{K}_{2n})}b^{3}_{\mathsf{e}}.
\end{align}

Now using \eqref{eq:asum}, we have
\begin{align}\label{eq:b33-1}
\sum\limits_{\substack{\mathsf{d},\mathsf{e},\mathsf{f} \in E(\mcal{K}_{2n})\\ \mathsf{d}\cap \mathsf{e}=\mathsf{e}\cap \mathsf{f}=\mathsf{d}\cap \mathsf{f}=\emptyset}} b_{\mathsf{d}}b_{\mathsf{e}}b_{\mathsf{f}} = &\sum\limits_{\substack{\mathsf{d}, \mathsf{e} \in E(\mcal{K}_{2n})\\ \mathsf{d}\cap \mathsf{e}= \emptyset}}b_{\mathsf{d}}b_{\mathsf{e}}  \left(\sum\limits_{\substack{\mathsf{f} \in E(\mcal{K}_{2n})\\ \mathsf{f}\cap (\mathsf{d} \cup \mathsf{e}) =\emptyset }} b_{\mathsf{f}} \right)\notag \\ 
\begin{split} 
 = & -\sum\limits_{\mathsf{d} \in  E(\mcal{K}_{2n})} \sum\limits_{\substack{ \mathsf{e} \in E(\mcal{K}_{2n})\\ \mathsf{d}\cap \mathsf{e}= \emptyset}} b_{\mathsf{d}}b_{\mathsf{e}} \left(\sum\limits_{t \in \mathsf{e} \cup \mathsf{d}} \left(\sum\limits_{\substack{ \mathsf{f} \in E(\mcal{K}_{2n})\\ t \in \mathsf{f}}} b_{\mathsf{f}} \right)\right) \\ 
     &+ \sum\limits_{\mathsf{d} \in  E(\mcal{K}_{2n})} \sum\limits_{\substack{ \mathsf{e} \in E(\mcal{K}_{2n})\\ \mathsf{d}\cap \mathsf{e}= \emptyset}} \sum\limits_{\mathsf{f} \in \binom{\mathsf{d} \cup \mathsf{e} }{ 2} } b_{\mathsf{d}}b_{\mathsf{e}}b_{\mathsf{f}} \notag \end{split}\\
= & \sum\limits_{\mathsf{d} \in  E(\mcal{K}_{2n})} \sum\limits_{\substack{ \mathsf{e} \in E(\mcal{K}_{2n}) \\ \mathsf{d}\cap \mathsf{e}= \emptyset}} \sum\limits_{\mathsf{f}\in \binom{\mathsf{d} \cup \mathsf{e} }{ 2}} b_{\mathsf{d}}b_{\mathsf{e}}b_{\mathsf{f}} \notag \\
= & \sum\limits_{\substack{\mathsf{d},\mathsf{e} \in E(\mcal{K}_{2n})\\ \mathsf{d} \cap \mathsf{e} =\emptyset}} (b^{2}_{\mathsf{d}}b_{\mathsf{e}}+b^{2}_{\mathsf{e}}b_{\mathsf{d}}) 
      + \sum\limits_{\substack{\mathsf{d},\mathsf{e} \in E(\mcal{K}_{2n})\\ \mathsf{d} \cap \mathsf{e} =\emptyset}} \sum\limits_{\mathsf{f} \in \binom{\mathsf{d} \cup \mathsf{e} }{ 2}\setminus \{\mathsf{d}, \mathsf{e}\}} b_{\mathsf{d}}b_{\mathsf{e}}b_{\mathsf{f}} \notag\\
= & 2\sum\limits_{\mathsf{e} \in E(\mcal{K}_{2n})} b^{3}_{\mathsf{e}} 
      + \sum\limits_{\substack{\mathsf{d},\mathsf{e} \in E(\mcal{K}_{2n})\\ \mathsf{d} \cap \mathsf{e} =\emptyset}} \sum\limits_{\mathsf{f} \in \binom{\mathsf{d} \cup \mathsf{e} }{ 2}\setminus \{\mathsf{d}, \mathsf{e}\}} b_{\mathsf{d}}b_{\mathsf{e}}b_{\mathsf{f}},
\end{align}
with the last equality following from \eqref{eq:b32}. In the second summand above, as $\mathsf{d}$ and $\mathsf{e}$ are disjoint $2$-sets, we have $\mathsf{f} \in \binom{\mathsf{d} \cup \mathsf{e} }{ 2}\setminus \{\mathsf{d}, \mathsf{e}\}$ if and only if $|\mathsf{f}\cap \mathsf{d}|= |\mathsf{f} \cap \mathsf{e}|=1$. With this observation, we have 
\begin{align}\label{eq:triplebs}
\sum\limits_{\substack{\mathsf{d},\mathsf{e} \in E(\mcal{K}_{2n})\\ \mathsf{d} \cap \mathsf{e} =\emptyset}} 
 \sum\limits_{\mathsf{f} \in \binom{\mathsf{d} \cup \mathsf{e} }{ 2}\setminus \{\mathsf{d}, \mathsf{e}\}} b_{\mathsf{d}}b_{\mathsf{e}}b_{\mathsf{f}} 
 & = \sum\limits_{\substack{\mathsf{d},\mathsf{f} \in E(\mcal{K}_{2n})\\ |\mathsf{d} \cap \mathsf{f}| =1}}b_{\mathsf{d}}b_{\mathsf{f}} \sum\limits_{\substack{\mathsf{e} \in E(\mcal{K}_{2n}) \\ \mathsf{e} \cap \mathsf{d}=\emptyset\ \&\ |\mathsf{e}\cap \mathsf{f}|=1}} b_{\mathsf{e}}.
\end{align}

For a fixed $\mathsf{d}=\{x,y\}$ and $\mathsf{f}=\{y,z\}$ withe $x,y,z$ distinct (so clearly $| \mathsf{d} \cap \mathsf{f} | = 1$), and any $\mathsf{e}$ with $\mathsf{e} \cap \mathsf{d} = \emptyset$ and $ |\mathsf{e}\cap \mathsf{f}|=1$,  using~\eqref{eq:asum} we have 
\[
\sum\limits_{\substack{\mathsf{e} \in E(\mcal{K}_{2n}) \\ \mathsf{e} \cap \mathsf{d}=\emptyset\ \&\ |\mathsf{e}\cap \mathsf{f}|=1}} b_\mathsf{e}=
\sum\limits_{\substack{\mathsf{e} \in E(\mcal{K}_{2n}), \\  z \in \mathsf{e},  \  x\notin \mathsf{e},  \ y \notin \mathsf{e}}} b_{\mathsf{e}}=-b_{\{y,z\}}-b_{\{x,z\}}.
\]
Using this equality in the inner summand on the right hand side of \eqref{eq:triplebs}, we have 
\begin{align*}
\sum\limits_{\substack{\mathsf{d},\mathsf{e} \in E(\mcal{K}_{2n})\\ \mathsf{d} \cap \mathsf{e} =\emptyset}} 
\sum\limits_{\mathsf{f} \in \binom{\mathsf{d} \cup \mathsf{e} }{ 2}\setminus \{\mathsf{d}, \mathsf{e}\}}   b_{\mathsf{d}}b_{\mathsf{e}}b_{\mathsf{f}} 
&= -\left[\sum\limits_{\substack{\mathsf{d},\mathsf{e},\mathsf{f} \in E(\mcal{K}_{2n})\\ 
   | \mathsf{d} \cap \mathsf{f} | =1,\ |\mathsf{d} \cap \mathsf{e}|=1, \ |\mathsf{e}\cap \mathsf{f}|=1}} b_{\mathsf{d}}b_{\mathsf{e}}b_{\mathsf{f}} \right] 
         - \sum\limits_{\substack{\mathsf{d},\mathsf{f} \in E(\mcal{K}_{2n})\\ |\mathsf{d} \cap \mathsf{f}| =1}}b_{\mathsf{d}}b^2_{\mathsf{f}} \\
& =-\left[\sum\limits_{\substack{\mathsf{d},\mathsf{e},\mathsf{f} \in E(\mcal{K}_{2n})\\ |\mathsf{d} \cap \mathsf{f}| =1,\ |\mathsf{d} \cap \mathsf{e}|=1, \ |\mathsf{e}\cap \mathsf{f}|=1}}b_{\mathsf{d}}b_{\mathsf{e}}b_{\mathsf{f}} \right] 
 +2 \sum\limits_{\mathsf{e} \in E(\mcal{K}_{2n})} b^{3}_{\mathsf{e}},
\end{align*}
with the last equality following from \eqref{eq:b32}. Using this equality along with~\eqref{eq:b31}, \eqref{eq:b32}, and~\eqref{eq:b33-1} yields the lemma.
\end{proof}

We now consider the second summand in the equation given by Lemma~\ref{lem:b33}.

\begin{lemma}
\label{lem:bdbebfBound}
For all $n$
\[
\sum\limits_{\substack{\mathsf{d},\mathsf{e},\mathsf{f} \in E(\mcal{K}_{2n})\\ |\mathsf{d} \cap \mathsf{f}| =1,\ |\mathsf{d} \cap \mathsf{e}|=1, \ |\mathsf{e}\cap \mathsf{f}|=1 }}b_{\mathsf{d}}b_{\mathsf{e}}b_{\mathsf{f}}
\geq 
 \frac{2c^{3}n(2n-2)}{3(2n-1)^{2}} -\frac{c^{3}n}{2n-1}  .
\]
\end{lemma}
\begin{proof}
We prove this by expanding the left-hand side by replacing each $b_{x}$, with $b_{x}= a_{x}- c/(2n-1)$, we have
\begin{align} \label{eq:b3extraterm}
\sum\limits_{\substack{\mathsf{d},\mathsf{e},\mathsf{f} \in E(\mcal{K}_{2n})\\ |\mathsf{d} \cap \mathsf{f}| =1, \  |\mathsf{d} \cap \mathsf{e}|=1, \ |\mathsf{e}\cap \mathsf{f}|=1}}b_{\mathsf{d}}b_{\mathsf{e}}b_{\mathsf{f}} 
      &= \Sigma_{1} 
           - \frac{c}{2n-1}  \Sigma_{2} 
           + \frac{c^{2}}{(2n-1)^2}\Sigma_{3}
           - \binom{2n }{ 3}\frac{c^{3}}{(2n-1)^{3}}, 
\end{align}
where
\begin{subequations}\label{eq:partsb3extraterm}
\begin{align}
\Sigma_{1} &=\sum\limits_{\{x,y,z\} \in \binom {V(\mcal{K}_{2n}) }{ 3}} a_{\{x,y\}}a_{\{x,z\}}a_{\{y,z\}}, \label{eq:s1def}\\
\Sigma_{2} &=\sum\limits_{\{x,y,z\} \in \binom {V(\mcal{K}_{2n})}{ 3}} a_{\{x,y\}}a_{\{x,z\}}+ a_{\{x,y\}}a_{\{y,z\}} +a_{\{x,z\}}a_{\{y,z\}},\label{eq:s2def}\\
\Sigma_{3} &=\sum\limits_{\{x,y,z\} \in \binom {V(\mcal{K}_{2n}) }{ 3}} a_{\{x,y\}}+ a_{\{x,z\}}+a_{\{y,z\}}.\label{eq:s3def}
\end{align}
\end{subequations}

As $a_{\mathsf{e}}\geq 0$ for all $\mathsf{e} \in E(\mcal{K}_{2n})$, we have 
\begin{equation}\label{eq:s1}
\Sigma_{1}>0.
\end{equation}
Now, we consider $\Sigma_{2}$. From \eqref{eq:s2def}, we have
\begin{align*}
\Sigma_{2} &= \sum\limits_{\{x,y,z\} \in \binom{V(\mcal{K}_{2n}) }{ 3}} a_{\{x,y\}}a_{\{x,z\}}+ a_{\{x,y\}}a_{\{y,z\}} +a_{\{x,z\}}a_{\{y,z\}} \\
& = \sum\limits_{\{x,z\} \in E(\mcal{K}_{2n})} \sum\limits_{y\in V(\mcal{K}_{2n}) \setminus \{x,z\}} a_{\{x,y\}}a_{\{y,z\}} \\
& =2 \sum\limits_{\{x,y\} \in E(\mcal{K}_{2n})}a_{\{x,y\}} \left(\sum\limits_{z\in V(\mcal{K}_{2n}) \setminus \{x,y\}} a_{\{y,z\}}\right) \\
& =2 \sum\limits_{\{x,y\} \in E(\mcal{K}_{2n})}a_{\{x,y\}}(c-a_{\{x,y\}})\ \ (\text{using \eqref{eq:asum}}) \\
& =2 c\sum\limits_{\{x,y\} \in E(\mcal{K}_{2n})}a_{\{x,y\}} - 2\sum\limits_{\{x,y\} \in E(\mcal{K}_{2n})}a^2_{\{x,y\}}.
\end{align*}
Now using the above and \eqref{eq:asum}, yields
\begin{equation}\label{eq:s2}
\Sigma_{2}  =c^2n- 2 \sum\limits_{\{x,y\} \in E(\mcal{K}_{2n})}a^2_{\{x,y\}} \leq c^2n.
\end{equation}

Moving onto $\Sigma_{3}$, noting each $2$-subset of $V(\mcal{K}_{2n})$ sits in exactly $2n-2$ of the $3$-subsets of $V(\mcal{K}_{2n})$, and using \eqref{eq:asum}, we observe that $\sum\limits _{\mathsf{e} \in E(\mcal{K}_{2n})} a_{\mathsf{e}} =c n$. Therefore, we have
\begin{equation}\label{eq:s3}
\Sigma_{3} = (2n-2)nc.
\end{equation}
Applying \eqref{eq:b3extraterm}, \eqref{eq:s1}, \eqref{eq:s2}, and \eqref{eq:s3} in Lemma~\ref{lem:b33} yields
the lemma.
\end{proof}

Inserting the bound from Lemma~\ref{lem:bdbebfBound} in Lemma~\ref{lem:b33} gives the following bound on the third moment of $h$.
\begin{lemma}\label{lem:thirdmomentbound}
For all $n$
\begin{align}\label{eq:b3}
\bb{E}[h^3] \leq &  \frac{2(n-2)}{(2n-3)(2n-5)} \sum\limits_{\mathsf{e} \in E(\mcal{K}_{2n})} b^{3}_{\mathsf{e}} + \frac{c^{3}n(2n+1)}{3(2n-1)^{3}(2n-3)(2n-5)}.
\end{align}
\end{lemma}

The next goal of this section is to find a lower bound on $\bb{E}[h^3]$. We recall that $g = \sum a_{\mathsf{e}} \ind_{\calS_{\mathsf{e}}}$ and using \eqref{eq:asum}, we see that  $h=g-\frac{n c}{2n-1}$ .

\begin{proposition}
\label{prop:gh}
For all $n$ 
\begin{subequations}\label{eq:affine}
\begin{align}
g &= \frac{2n-2}{2n-3} f_{1}  + \frac{c(n-2)}{(2n-3)} \one, \text{ and} \label{eq:gaffine}\\
h &= \frac{2n-2}{2n-3} f_{1}  - \frac{c(2n-2)}{(2n-1)(2n-3)} \one \label{eq:haffine}.
\end{align}
\end{subequations}
\end{proposition}
\begin{proof}
Recall $U$ is the space spanned by the characteristic vectors of the stars, and, by definition, both $g$, $f_1$ are contained in $U$, as is the all ones vector. So, to show the first statement in this proposition, we will show that   
\[
g - \frac{2n-2}{2n-3} f_{1}  + \frac{c(n-2)}{(2n-3)} \one
\]
is orthogonal to every $\ind_{\calS_x}$, and hence is equal to 0. 

First, for any edge $\mathsf{e}$, we have both
\[
\left\langle f_{1}, \ind_{\calS_{\mathsf{e}}} \right\rangle = \left\langle \ind_A, \ind_{\calS_{\mathsf{e}}} \right\rangle=\frac{a_{\mathsf{e}}}{2n-1}, \quad
\left\langle \one, \ \ind_{\calS_{\mathsf{e}}} \right\rangle= \frac{1}{2n-1}.
\] 

Next, consider the edge $\mathsf{e} = \{x,y\}$ and $\calS_{\mathsf{e}}$, using~\eqref{eq:pmstarint} and~\eqref{eq:asum}, we have 
\begin{align*}
\left\langle g,\ \ind_{\calS_{\mathsf{e}}} \right\rangle &= \dfrac{a_{\mathsf{e}} }{2n-1}+ \sum\limits_{\{\mathsf{f} \ :\ \mathsf{f} \ \cap \{x,y\}=\emptyset\}} \dfrac{a_{\mathsf{f}}}{(2n-1)(2n-3)}\\
& = \dfrac{a_\mathsf{e}  }{2n-1} + \dfrac{1}{2 (2n-1)(2n-3) } \sum_{z\neq x,y}
    \left( \sum\limits_{ \{\mathsf{f} \ : \ z\in \mathsf{f}, \mathsf{f} \cap \{x,y\}=\emptyset\}} a_{\mathsf{f}}  \right) \\
&= \dfrac{a_{\mathsf{e}} }{2n-1} + \dfrac{1}{2 (2n-1)(2n-3) } \sum_{z\neq x,y}c-a_{\{x,z\}}-a_{\{y,z\}} \\
& = \dfrac{a_{\mathsf{e}} }{2n-1} + \dfrac{1}{2 (2n-1)(2n-3) } \left( (2n-2)c+a_{\{x,y\}}-c+a_{\{x,y\}}-c \right)\\
&= \frac{1}{2n-1} \left ( a_{\mathsf{e}}\frac{2n-2}{2n-3} + \frac{c(n-2)}{2n-3}\right).
\end{align*}

The result follows from expanding $\left\langle g - \frac{2n-2}{2n-3} f_{1}  + \frac{c(n-2)}{(2n-3)} \one, \ind_{\calS_\mathsf{e}} \right\rangle$.
The second statement follows as $h = \sum\limits_{\mathsf{e} \in E(\mcal{K}_{2n})}   g - \frac{c}{2n-1} \ind_{\calS_\mathsf{e}}$ and $\sum\limits_{\mathsf{e} \in E(\mcal{K}_{2n})} \ind_{\calS_\mathsf{e}} = n \one$.
\end{proof}

From this proposition and the fact that $\bb{E}[f_{1}]=\bb{E}[f]=\frac{c}{2n-1}$, we have 
\begin{equation}\label{eq:1mh}
\bb{E}[h]=0. 
\end{equation}

We can now give the value of the second moment of $h$, provided that $\bb{E}[(f-f_{1})^{2}]$ is ``small". 

\begin{lemma}\label{lem:2mh}
Let $\varepsilon$ be such that $\bb{E}[(f-f_{1})^{2}]=\dfrac{\varepsilon c}{2n-1}$, then
\[
\bb{E}[h^2] = \frac{c(1-\varepsilon)(2n-2)^2}{(2n-3)^2(2n-1)}- \frac{c^{2}(2n-2)^2}{(2n-3)^2(2n-1)^2}.
\]
\end{lemma}
\begin{proof}
Since $f = \ind_A$, is a $0$--$1$ vector, so $\bb{E}[f] = \bb{E}[f^2]$ and we have seen that $\bb{E}[f_{1}]=\bb{E}[f]=\frac{c}{2n-1}$.
As $f_{1}$ is orthogonal to $f-f_{1}$, we further have 
\begin{equation}\label{eq:eps<1}
\bb{E}[f^{2}_{1}]=\bb{E}[f^2]-\bb{E}[(f-f_{1})^2]=\bb{E}[f]-\frac{\varepsilon c}{2n-1}=\frac{c(1-\varepsilon)}{2n-1}.
\end{equation} 
The result now follows by using \eqref{eq:haffine} and \eqref{eq:1mh}.
\end{proof}

Before we proceed further, we observe that $\epsilon< 1$, by design. This follows from the fact that the left hand side of \eqref{eq:eps<1} is positive.  
Lemma~\ref{lem:b2}, and Lemma~\ref{lem:2mh} give the bound
\begin{equation}\label{eq:b2boundint}
\sum\limits_{\mathsf{e} \in E(\mcal{K}_{2n})} b^{2}_{\mathsf{e}} \leq  \frac{(2n-2)}{(2n-3)}c- \frac{(2n-2)}{(2n-3)}\frac{c^{2}}{(2n-1)}.
\end{equation}

We are working to find a lower bound on $\bb{E}[h^3]$, and the reminder of this section follows the method in~\cite{EFF2015} very closely. We include the details for completeness.

The next step is to give a lower bound on $\bb{E}[f_1^3]$, and then use this bound in Proposition~\ref{prop:gh}. We have established the following properties of $f$:
\begin{enumerate}
\item $\bb{E}(f_1) = \bb{E}(f) = \frac{c}{2n-1}$, and
\item $f$ takes only two values, 0 and 1.
\end{enumerate}
We will further include the assumption from Theorem~\ref{thm:EFFPM}, that
\begin{enumerate}[\indent3.]
\item $\bb{E}[(f-f_{1})^{2}]=\dfrac{\varepsilon c}{2n-1}$. 
\end{enumerate}

To find a lower bound on $\bb{E}[ f^3 ]$, we will apply Lemma~\ref{lem:optim}, which is an optimization result derived in~\cite{EFF2015}.  
This result uses the piecewise continuous function $\Phi:[0,1] \to \bb{R}_{\geq 0}$ defined by  $\Phi :=  1 \cdot \ind_{[0,\ \theta)}+ 0 \cdot \ind_{[\theta,\ 1]}$ with $\theta = \dfrac{c}{2n-1}$.  We can identify any discrete function $\phi$ on $V(\Ma_n)$ with a piecewise-continuous function defined on $[0,1]$. Simply, identify the $i^{th}$ element of $V(\Ma_n)$ with the number $\frac{i}{(2n-1)!!} \in [0,1]$, and for $x \in \left( \frac{i-1}{(2n-1)!! } ,  \frac{i}{(2n-1)!!} \right]$ define $\phi'(x) = \phi(i)$ (and set $\phi'(0) = \phi(1)$). Clearly for any discrete $\phi$ on $V(\Ma_n)$, $\bb{E}[\phi] = \bb{E}[\phi']$.
If we assume the elements of $V(\Ma_n)$ are ordered so the elements in $A$ occur first, then $f' = \Phi$.

We will find a lower bound on $\bb{E}[f_1^3 ]$, by considering the piecewise continuous function $f_1'$ and using the following optimization lemma.

\begin{lemma}\label{lem:optim}
Let $\theta \in (0,1)$ and let $H,L,\eta\in \bb{R}_{\geq 0}$ be such that $H>L$ and 
\[
\dfrac{\eta}{(\theta)(1-\theta)}\leq (H-L)^2.
\] 
Let $\Phi := H \cdot \ind_{[0,\ \theta)}+ L \cdot \ind_{[\theta,\ 1]}$. If $\phi:[0,1] \to \bb{R}_{\geq 0}$ is a measurable function such that $\bb{E}[ \phi]=\bb{E}[\Phi]$ and $\bb{E}[ (\phi-\Phi)^2] \leq \eta$, then 
\[
\bb{E}[ \phi^3 ] \geq \theta H^3 +(1-\theta)L^{3}-3(H^2-L^2)\sqrt{\theta(1-\theta)\eta} 
           +3 \left( (1-\theta)L+\theta H \right) \eta -\dfrac{1-2\theta}{\sqrt{\theta(1-\theta)}} \eta^{3/2}. 
\]
\end{lemma}

Since $\bb{E} [ (f_1)^3 ]  = \bb{E} [ (f_1')^3 ]$, we will apply the above lemma with $\phi = f_1'$. We will use the parameters
\[
 \theta= \dfrac{c}{2n-1}, \quad  H=1, \quad L=0,
 \] 
and show that $f_1'$ satisfies the conditions for $\phi$.
First, by definition $\bb{E}[\Phi] = \theta = \dfrac{c}{2n-1}$ and $\bb{E}[f_1'] = \bb{E}[f_1] =  \bb{E}[f] =  \dfrac{c}{2n-1}$, so we have $\bb{E}[\Phi] = \bb{E}[f_1']$.
Next, we have assumed that $\bb{E}[(f-f_{1})^{2}]=\dfrac{\varepsilon c}{2n-1}$, and $\Phi = f'$, so
\[
\bb{E}[(\Phi -f_{1}')^{2}]= \bb{E}[(f-f_{1})^{2}]=\dfrac{\varepsilon c}{2n-1},
\]
thus we set $\eta = \dfrac{\varepsilon c}{2n-1}$.
Finally, we will make the additional assumption that $\varepsilon \leq \frac{1}{2}$, as this implies $\varepsilon^{3/2} \leq  \varepsilon^{1/2}$,
and that
\[
\frac{\eta}{\theta (1-\theta)}  = \frac{\dfrac{\varepsilon c}{2n-1} }{\dfrac{c}{2n-1} \left(1-\dfrac{c}{2n-1} \right) } = \frac{\frac12 (2n-1) }{2n-1-c}  \leq 1 = (H-L)^2
\]
since $c \leq (2n-1)/2$. Now we can apply Lemma~\ref{lem:optim} with these parameters to get
\begin{align*}
\bb{E}[f^{3}_{1}] 
\geq & \frac{c}{2n-1} -3 \frac{c \sqrt{\varepsilon}}{2n-1} \sqrt{1-\frac{c}{2n-1}} + \frac{3 \varepsilon c^2}{(2n-1)^2} -\frac{1-\frac{2c}{2n-1}}{\sqrt{1-\frac{c}{2n-1}}} \frac{c \varepsilon^{3/2}}{2n-1} \notag \\
\geq & \frac{c}{2n-1} -3 \frac{c \sqrt{\varepsilon}}{2n-1}-  \frac{c \sqrt{\varepsilon}}{2n-1} \\
= & \frac{c}{2n-1} - \frac{4c \sqrt{\varepsilon}}{2n-1}.
\end{align*}  

Using \eqref{eq:haffine} and \eqref{eq:1mh}, we have 
\begin{align}\label{eq:f3h3}
\begin{split}
\frac{(2n-2)^{3}}{(2n-3)^{3}}\bb{E}[f^{3}_{1}] =& \bb{E}[h^3]  + \frac{3c(2n-2)}{(2n-1)(2n-3)}\bb{E}[h^2] \\
 & + \frac{3c^{2}(2n-2)^{2}}{(2n-1)^{2}(2n-3)^{2}}\bb{E}[h]  \frac{c^{3}(2n-2)^{3}}{(2n-1)^{3}(2n-3)^{3}} 
\end{split}\notag\\
= & \bb{E}[h^3] + \frac{3c(2n-2)}{(2n-1)(2n-3)}\bb{E}[h^2] + \frac{c^{3}(2n-2)^{3}}{(2n-1)^{3}(2n-3)^{3}}.
\end{align}

The above result along with Lemma~\ref{lem:2mh} and \eqref{eq:f3h3} yields
\begin{align}\label{eq:intermidiateineq}
\bb{E}[h^3] \geq & \frac{(2n-2)^{3}}{(2n-3)^{3}} \bb{E}[f^{3}_{1}] -\frac{c^{3}(2n-2)^{3}}{(2n-1)^{3}(2n-3)^{3}} -\frac{3c(2n-2)}{(2n-1)(2n-3)} \bb{E}[h^{2}] 
\notag \\ 
\begin{split}
  \geq & \frac{(2n-2)^{3}}{(2n-3)^{3}}\left[\frac{c}{2n-1} - \frac{4c \sqrt{\varepsilon}}{2n-1}\right] -\frac{c^{3}(2n-2)^{3}}{(2n-1)^{3}(2n-3)^{3}}\\ & + \frac{3c^{3}(2n-2)^{3}}{(2n-1)^{3}(2n-3)^{3}} +\frac{3 c^{2}\varepsilon (2n-2)^{3}}{(2n-3)^{3}(2n-1)^{2}} - \frac{3 c^{2} (2n-2)^{3}}{(2n-3)^{3}(2n-1)^{2}} \notag
\end{split} \\
\geq & \frac{(2n-2)^{3}}{(2n-3)^{3}}\left(\frac{c}{2n-1} - \frac{4c \sqrt{\varepsilon}}{2n-1}  - \frac{3c^{2}}{(2n-1)^2}  +\frac{2c^{3}}{(2n-1)^3} \right). 
\end{align}

Using \eqref{eq:intermidiateineq} along with \eqref{eq:b3} yields
\begin{align}\label{eq:b3boundint}
\begin{split}
\sum \limits_{e \in E(\mcal{K}_{2n})} b^{3}_{e} 
 \geq & \frac{(2n-5)(2n-2)^{3}}{(2n-4) (2n-3)^2(2n-1)}c - 4\frac{(2n-5)(2n-2)^{3}}{(2n-4)(2n-3)^{2}(2n-1)} c \sqrt{\varepsilon} \\  
 & -\left(\frac{3(2n-2)^{3}(2n-5)}{(2n-4)(2n-3)^2(2n-1)} \right)\frac{c^{2}}{2n-1} \\
 & +\left(\frac{2(2n-5)(2n-2)^{3}}{(2n-4)(2n-3)^2(2n-1)} -\frac{n(2n+1)}{3(2n-4)(2n-1)} \right) \frac{c^{3}}{(2n-1)^{2}}.
\end{split}
\end{align}

We now observe the following:
\begin{enumerate}[a)]
\item The difference between the coefficient of $c$ in \eqref{eq:b2boundint} and the above equation is\\ $O(1/n) + O(1) \sqrt{\epsilon}$;
\item the difference between the coefficient of $c^{2}/(2n-1)$ in \eqref{eq:b2boundint} and the above equation tends to $2$ as $n\to \infty$; and
\item the coefficient of $c^{3}/(2n-1)^2$ is positive as it tends to $11/6$ as $n \to \infty$.
\end{enumerate}

Therefore, we have
\begin{subequations}
\begin{align}
\sum\limits_{\mathsf{e} \in E(\mcal{K}_{2n})} b^{3}_{\mathsf{e}} \geq c-O\left(\frac{c}{n} \right)-O(c \sqrt{\varepsilon})-O\left(\frac{c^{2}}{n} \right), \label{eq:b3-0} \\
\sum\limits_{\mathsf{e} \in E(\mcal{K}_{2n})} b^{2}_{\mathsf{e}}-b^{3}_{\mathsf{e}} \leq O(c \sqrt{\varepsilon})+O\left(\frac{c}{n}\right)+O\left(\frac{c^{2}}{n} \right).\label{eq:b2-b3-0}
\end{align} 
\end{subequations}

We will show that these equations prove that $A$ is close to being a union of $\mrm{round}(c)$ stars, the first step is to show that these two equations force $c$ to not be too small.

\begin{lemma}\label{lem:clb}
We have
\[
c^{1/2} \geq 1- O \left( \frac{1}{n} \right)- O ( \sqrt{\varepsilon}).
\]
\end{lemma}
\begin{proof}
The statement is clearly true for any $c>1$, so without loss of generality assume $c \leq 1$.
 From \eqref{eq:b2boundint}, we have for $n$ sufficiently large,
\begin{equation}\label{eq:b2bound}
\sum\limits_{\mathsf{e} \in E(\mcal{K}_{2n})}b^{2}_{\mathsf{e}} \leq c +\frac{c}{2n-3} \leq c\dfrac{2n-2}{2n-3} \leq c \left( 1+ O\left(\frac{1}{n} \right) \right).
\end{equation} 
Now by monotonicity of $p$-norms, we have
\begin{align*}
\sum\limits_{\mathsf{e} \in E(\mcal{K}_{2n})} b^{3}_{\mathsf{e}} \leq \left(\sum\limits_{\mathsf{e} \in E(\mcal{K}_{2n})} b^{2}_{\mathsf{e}} \right)^{3/2}
      \leq c^{3/2} \left( 1+O \left( \frac{1}{n} \right) \right).
\end{align*}

Using this, with~\eqref{eq:b3-0}, we have
\begin{align*}
c^{3/2}(1+O \left( \frac{1}{n} \right) \geq c -O \left( \frac{c}{n} \right)- O(c\sqrt{\varepsilon}) -O \left( \frac{ c^{2} }{ n} \right), 
\end{align*}
and thus 
\[
c^{1/2} \geq 1- O \left( \frac{1}{n} \right)-O(\sqrt{\varepsilon}).
\]
\end{proof}

This result can be used to simplify the bounds in~\eqref{eq:b3-0} and \eqref{eq:b2-b3-0}.

\begin{cor}
Provide that $n$ is sufficiently large and $\varepsilon$ is sufficiently small, then
\begin{subequations}\label{eq:basymbound}
\begin{align}
\sum\limits_{\mathsf{e} \in E(\mcal{K}_{2n})} b_{\mathsf{e}}^{3} & \geq c - O(c^2 \sqrt{\varepsilon})-O\left(\frac{c^2}{n}\right),\label{eq:b3bound} \\
\sum\limits_{\mathsf{e} \in E(\mcal{K}_{2n})} b^{2}_{\mathsf{e}}-b^{3}_{\mathsf{e}} &\leq O(c^2 \sqrt{\varepsilon})+O\left(\frac{c^2}{n}\right). \label{eq:b2-b3} 
\end{align} 
\end{subequations}
\end{cor}
\begin{proof}
For $n$ sufficiently large and $\varepsilon$ is sufficiently small, using Lemma~\ref{lem:clb}, we have $c\geq \frac{1}{2}$. Thus we can use the inequality $c \leq 2c^{2}$ in~\eqref{eq:b3-0} and \eqref{eq:b2-b3-0} to yield the corollary. 
\end{proof}

Our goal is to show that each of the parameters $b_\mathsf{e}$ is either close to 1, or close to 0. To do this, define the sequence $( x_{1},\ x_{2},\ \ldots x_{\binom{2n }{ 2}})$ be the arrangement of the parameters $b_{\mathsf{e}}$ in non-increasing order. We will show that the first $\mrm{round}(c)$ elements of the sequence are close to $1$. Using~\eqref{eq:b2-b3}, proceeding as in \S~3 of~\cite{EFF2015}, we obtain a lower bound on the sum of the first $c$ elements in the sequence. 
The proof of this is essentially same as that of~\cite[Equation $(39)$]{EFF2015}. We include it for completeness.

\begin{lemma}\label{lem:sumb}
Provided $n$ is large enough and $\varepsilon$ is small enough, we have
\begin{align*}
\sum\limits_{i=1}^{\mrm{round}(c)} x_{i} &\geq \mrm{round}(c) - O(c^2 \sqrt{\varepsilon}) -  O\left( \frac{ c^{2} }{ n} \right) , \\
 |c-\mrm{round}(c)| & \leq O(c^2 \sqrt{\varepsilon} ) + O \left( \frac{ c^{2} }{n} \right).
\end{align*}
\end{lemma}
\begin{proof}
By the definition of the parameters $b_{\mathsf{e}}$, we have $-\frac{c}{2n-1} \leq x_{i} \leq 1-\frac{c}{2n-1}$, for all $1\leq i \leq \binom{2n }{ 2}$.
So we can set $m$ to be the largest index $i$ such that $x_{i} \geq \frac{1}{2}$. The first step in this proof is to show that the equation holds when $\mrm{round}(c)$ is replaced with $m$.

As $x_{j} \leq 1/2$ for $j > m $, we have
\begin{align*}
m \geq & \sum\limits_{i\leq m} x^2_{i} \\
 \geq & \sum\limits_{i=1}^{\binom{2n }{ 2}} x^{2}_{i}- \sum\limits_{j>m}x^2_{j}\\
 \geq & \sum\limits_{i=1}^{\binom{2n }{ 2}} x^{2}_{i} -2\sum\limits_{j>m}x^{2}_{j}(1-x_{j})\\
 \geq & \sum\limits_{i=1}^{\binom{2n }{ 2}} x^{2}_{i} -2\sum\limits_{i=1}^{\binom{2n }{ 2}}(x^{2}_{j}-x^{3}_{j})\\
 = & \sum\limits_{i=1}^{\binom{2n }{ 2}} x^{3}_{i} - \sum\limits_{i=1}^{\binom{2n }{ 2}}(x^{2}_{j}-x^{3}_{j}).
\end{align*}
Now using~\eqref{eq:b3bound} and \eqref{eq:b2-b3}, we have 
\begin{equation}\label{eq:mlb}
m \geq c- O(c^{2} \sqrt{\varepsilon}) - O \left( \frac{ c^{2} }{ n} \right)
\end{equation}

Further, \eqref{eq:b2-b3} and the fact that $1 \leq 4x^2_{i}$ for all $i\leq m$, imply
\begin{align*}
m-\sum\limits_{i=1}^{m}x_{i} = \sum\limits_{i=1}^{m}  1 (1-x_{i}) 
                                           \leq 4 \sum\limits_{i=1}^{m} x^2_{i}(1-x_{i}) 
                                           \leq O(c^2 \sqrt{\varepsilon} ) +O \left( \frac{c^{2}}{n} \right),
\end{align*}
thus, 
\begin{equation}\label{eq:sumlb}
\sum\limits_{i=1}^{m}x_{i}  \geq m-O(c^2 \sqrt{\varepsilon} ) -O \left( \frac{c^{2} }{ n} \right).
\end{equation}

Since $x_{i} \geq 1/2$ for all $i\leq m$, we have $x^2_{i} \geq 2x_{i}-1$. Now, using~\eqref{eq:b2bound} and \eqref{eq:sumlb}, we have 
\[
c + O \left( \frac{c}{n} \right) 
\geq  \sum\limits_{i=1}^{m}x^{2}_{i} 
\geq 2 \left( \sum\limits_{i=1}^{m}x_{i} \right) -m 
\geq m -O \left( c^2 \sqrt{\varepsilon} \right) -O \left( \frac{ c^{2} }{ n} \right).
\] 
Combining the above inequality with~\eqref{eq:mlb}, we have 
\begin{equation}\label{eq:c-m}
|c-m| \leq O(c^2 \sqrt{\varepsilon}) +O\left( \frac{c^{2}}{n} \right).
\end{equation}

We will now show that~\eqref{eq:sumlb} is true upon replacing $m$ with $\mrm{round}(c)$.
Next we will show that the equation holds when $m$ is replaced with an integer $m'$ with $|c-m'|<1$. We split this into two cases.

{\bf Case 1: $c \leq m$. }

We set $m'=\lceil c \rceil$, $c\leq m' \leq m$. As $(x_{i})_{i=1}^{\binom{2n}{2}}$ is a non-increasing sequence, we have $\frac{1}{m}\sum\limits_{i=1}^{m} x_{i} \leq \frac{1}{m'} \sum\limits_{i=1}^{m'} x_{i}$. Therefore by~\eqref{eq:sumlb}, we have
\[
\sum\limits_{i=1}^{m'} x_{i} 
\geq \frac{m'}{m}\left(  m-O(c^2 \sqrt{\varepsilon} ) - O\left( \frac{c^{2}}{n} \right)\right) \geq  m'- O(c^2 \sqrt{\varepsilon} ) -O\left( \frac{c^{2}}{n} \right).
\]

{\bf Case 2: $c > m$.} 

In this case, set $m'=\lfloor c \rfloor$, so $m \leq m' \leq c$. As $x_{i} \geq -\frac{c}{(2n-1)}$, using \eqref{eq:sumlb}, we have 
\begin{align*}
\sum\limits_{i=1}^{m'} x_{i} &
\geq \sum\limits_{i=1}^{m} x_{i} +(m'-m)\frac{-c}{(2n-1)}\\ 
&\geq \left(  m-O(c^2 \sqrt{\varepsilon} ) -O\left( \frac{c^{2}}{n} \right) \right)  -\frac{m'c}{2n-1}\\
&\geq m'-(m'-m) -O(c^2 \sqrt{\varepsilon} ) -O\left( \frac{c^{2}}{n} \right) -\frac{m'c}{2n-1}\\
&\geq m'-(c-m) -O(c^2 \sqrt{\varepsilon} ) -O\left( \frac{c^{2}}{n} \right) -\frac{c^2}{2n-1}, \ \text{since $m'\leq c$ }\\
& \geq m'-O(c^2 \sqrt{\varepsilon} ) -O\left( \frac{c^{2}}{n} \right),
\end{align*}
with the last inequality following from~\eqref{eq:c-m}.

In both cases, $m'$ is between $c$ and $m$, so~\eqref{eq:c-m} is true even after replacing $m$ with $m'$; so, in both cases, we have an integer $m'$ which satisfies 
\begin{subequations}
\begin{align}
|c-m'| & < \mrm{min} \left\{1,\ O(c^2 \sqrt{\varepsilon} )+O\left( \frac{c^{2}}{n} \right) \right\}   \label{eq:c-m'}\\
\sum\limits_{i=1}^{m'} x_{i} & \geq m'-O(c^2 \sqrt{\varepsilon} ) -O\left( \frac{c^{2}}{n} \right). \label{eq:m'lb}
\end{align}
\end{subequations}

Finally, we need to show that~\eqref{eq:sumlb} holds when $m$ is replaced by $\mrm{round}(c)$. If $\mrm{round}(c) = m'$, then both statements in the lemma hold true (the second statement follows from~\eqref{eq:c-m'}).
So we need to consider the cases $\mrm{round}(c)=m'+1$ and $\mrm{round}(c)=m'-1$, in either case~\eqref{eq:c-m'} implies
\begin{equation}\label{eq:1lb}
1=|m'-\mrm{round}(c)|\leq 2|m'-c| \leq O(c^2 \sqrt{\varepsilon} ) +O\left( \frac{c^{2}}{n} \right).
\end{equation}

{\bf Case a. $\mrm{round}(c)=m'+1$}

Since $x_{\mrm{round}(c)} \geq -c/(2n-1)$, using~\eqref{eq:m'lb} and \eqref{eq:1lb}, we have 
\begin{align*}
\sum\limits_{i=1}^{\mrm{round}(c)}x_{i} & \geq \sum\limits_{i=1}^{m'}x_{i} -\frac{c}{(2n-1)} \\
& \geq m'-O(c^2 \sqrt{\varepsilon} ) -O\left( \frac{c^{2}}{n} \right)\\
& \geq \mrm{round}(c)- O(c^2 \sqrt{\varepsilon} ) -O\left( \frac{c^{2}}{n} \right).
\end{align*} 
The second inequality is true as by Lemma~\ref{lem:clb}, we have $c \leq 2c^2$ for $n$ large enough and $\varepsilon$ small enough. 

{\bf Case b. $\mrm{round}(c)=m'-1$}

As $x_{m'}\leq 1$, we have 
\begin{align*}
\sum\limits_{i=1}^{\mrm{round}(c)}x_{i} & \geq \sum\limits_{i=1}^{m'}x_{i} -1 \\
& \geq m'-1-O(c^2 \sqrt{\varepsilon} ) - O\left( \frac{c^{2}}{n} \right)\\
& = \mrm{round}(c)- O(c^2 \sqrt{\varepsilon} ) - O\left( \frac{c^{2}}{n} \right).
\end{align*} 

The second statement in the lemma holds from~\eqref{eq:c-m'} and the fact that $|c - \mrm{round}(c) | \leq | c -m'|$.
\end{proof}

We are now ready to prove Theorem~\ref{thm:EFFPM}. 

\begin{proof}[Proof of Theorem~\ref{thm:EFFPM}]
Let $Z$ be the $\mrm{round}(c)$-set of edges in $\mcal{K}_{2n}$ such that 
$\{b_{\mathsf{e}}\ :\ \mathsf{e} \in Z\}=\{x_{i}:\ i \in [1, \dots, \mrm{round}(c)]\}$. Using the definition of the parameters $x_{i}$ and Lemma~\ref{lem:sumb}, we have
\begin{align*}
\sum\limits_{\mathsf{e} \in Z}|A \cap \calS_{\mathsf{e}}| & \geq (2n-3)!! \sum\limits_{i=1}^{\mrm{round}(c)} x_{i} \\
 & \geq (2n-3)!!\left(\mrm{round}(c)- O(c^2 \sqrt{\varepsilon} )-O\left( \frac{c^{2}}{n} \right)\right).
\end{align*}

Set $B= \bigcup\limits_{\mathsf{e} \in Z} \calS_{\mathsf{e}}$. 
By \eqref{eq:pmstarint}, for any two edges $\mathsf{e},\mathsf{f}$, $|\calS_{\mathsf{e}} \cap \calS_{\mathsf{f}}| \leq (2n-5)!!$, thus,
\begin{align*}
|A \cap B| &\geq \sum\limits_{\mathsf{e} \in Z} |A \cap \calS_{\mathsf{e}}| - \binom{\mrm{round}(c)}{2} (2n-5)!!\\ 
   & \geq (2n-3)!! \left( \mrm{round}(c)- O(c^2 \sqrt{\varepsilon} )  - O\left( \frac{c^{2}}{n} \right) \right) - O\left( \frac{c^2}{n} \right) (2n-3)!!\\ 
   & \geq (2n-3)!! \left( \mrm{round}(c)- O(c^2 \sqrt{\varepsilon} )  - O\left( \frac{c^{2}}{n} \right) \right).
\end{align*}
As $|A|=c(2n-3)!!$ and $|B|=\mrm{round}(c)(2n-3)!!$, we have 
\begin{align*}
A\Delta B &= |A|+|B|-2|A\cap B|  \\
  & \leq c(2n-3)!! +\mrm{round}(c)(2n-3)!! - (2n-3)!! \left( \mrm{round}(c)- O(c^2 \sqrt{\varepsilon} )  - O\left( \frac{c^{2}}{n} \right) \right)\\
&=   \left( O(c^2 \sqrt{\varepsilon} ) +O\left( \frac{c^{2}}{n} \right) \right).
\end{align*}
where the last equation uses Lemma~\ref{lem:sumb}, that is $|c-\mrm{round}(c)| \leq O(c^2 \sqrt{\varepsilon} ) + O\left( \frac{c^{2}}{n} \right)$. 

This proves our theorem when $n$ is sufficiently large, for a choice of sufficiently small $\varepsilon_{0}$. For a fixed $n_{0}$ and $\varepsilon_{0}$, we can choose a sufficiently large absolute constant $C_{0}$ so that the result is true for all $n\leq n_{0}$.
\end{proof}


\section{Proof of Theorem~\ref{thm:pmweakstability}}

Let $\calS$ be an independent set in $G= \Ma_n$ of size $c(2n-3)!!$ with $c\leq 1$. Let $k,\tau,\mu$ be as in Proposition~\ref{prop:edge-count-norm}. We have from Lemma~\ref{lem:specpmgraph} and \eqref{eq:numderpm} that 
\[
-\tau=\frac{-\de_{n}}{2(n-2)}=\frac{(2n-1)!!}{2n-2}\left( \frac{1}{\sqrt{e}} +o(1)\right), \quad -\mu=O((2n-5)!!).
\]
Together these imply $\mu-\tau \geq | \tau |/2$.
Since $\calS$ is independent, we have $\ed(\calS)=0$, so applying Lemma~\ref{prop:edge-count-norm}, 
we have 
\begin{align*}
\|\ind_{\calS}-\operatorname{Proj}_{U}(\ind_{\calS})\|^{2} 
\leq \frac{c(1-c)}{(2n-1)} \frac{|\tau|}{\mu-\tau}
\leq \frac{2(1-c)c}{(2n-1)}.
\end{align*}

By Theorem~\ref{thm:EFFPM}, there are absolute constants $C_{0}$ and $\kappa\leq 1$ such that for any independent set $\calS$ of size $c(2n-3)!!$ with $c\geq \kappa$, there is an edge $\mathsf{f} \in E(\mcal{K}_{2n})$ such that 
\[
|\calS \Delta \calS_{\mathsf{f}}| \leq C_{0}c^{2}(2n-3)!!\left(\sqrt{(2(1-c)} + \frac{1}{2n-1}\right).
\]  

Given $\phi>0$, we can pick a $c(\phi)<1$ such that for any $c \in [c(\phi), 1]$, we have
\[
C_{0}\left(\sqrt{(2(1-c)} + \frac{1}{2n-1}\right)< \phi,
\]
for all $n\geq 3$.
We have now proved the following.
\begin{lemma}\label{lem:symdifub}
Given $\phi>0$, there exists a $c(\phi)$ such that the following holds. If $\calS$ is an independent set in $\Ma_{n}$ of size $c(2n-3)!!$ with $1\geq c \geq c(\phi)$, then there is an edge $\mathsf{f}$ in $\mcal{K}_{2n}$ such that \[\calS \Delta \calS_{\mathsf{f}} \leq ((2n-3)!!)\phi.\]
\end{lemma}

\begin{proof}[Proof of Theorem~\ref{thm:pmweakstability}]
Pick some $0<\phi<\frac{1}{\sqrt{e}}$, and $c(\phi)$ be as in Lemma~\ref{lem:symdifub}. Then for any independent set of size $c(2n-3)!!$ with $c\geq c(\phi)$, there is an edge $\mathsf{f}$ of $\mcal{K}_{2n}$ such that 
\begin{equation}\label{eq:phi}
\calS \Delta \calS_{\mathsf{f}} \leq (2n-3)!! \phi.
\end{equation}

We will now show that the above inequality implies $\calS \subset \calS_{\mathsf{f}}$ when $n$ is sufficiently large. We assume the contrary, that is, there is a perfect matching $\mcal{P} \in \calS$ that does not contain the edge $\mathsf{f}$. Using Theorem~\ref{thm:ratiobound}, $\mathcal{P}$ has $-\tau$ neighbours in $\calS_{\mathsf{f}}$. Therefore, since $\calS$ is an independent set, we have 
\begin{equation}\label{eq:intlb}
|\calS_{\mathsf{f}} \setminus \calS| \geq \frac{\de_{n}}{(2n-2)}.
\end{equation}
From \eqref{eq:phi} and \eqref{eq:intlb}, we arrive at the following contradiction 
\[
\frac{1}{\sqrt{e}}> \phi> \frac{\calS \Delta \calS_{\mathsf{f}} }{(2n-3)!! } \geq \frac{|\calS_{\mathsf{f}} \setminus \calS|  }{(2n-3)!!} \frac{\de_{n}}{(2n-3)!!(2n-2)}>\frac{\de_{n}}{(2n-1)!!}=  \frac{1}{\sqrt{e}}+o(1) .
\] 
Thus our assumption is false.
\end{proof}

\bibliographystyle{plain}
\bibliography{ref}            
\end{document}